\documentclass[a4paper,11pt,reqno]{article}  
\usepackage{amsmath,amsthm,amssymb,amscd,mathdots,stmaryrd,enumerate,shuffle,ytableau,varwidth,faktor,xfrac,xcolor,calc,titlepic,setspace,enumitem,csquotes}
\usepackage[utf8]{inputenc}
\usepackage[nottoc]{tocbibind}
\usepackage{pst-node}
\usepackage{auto-pst-pdf}
\usepackage{tikz-cd} 
\usepackage{tikz}
\usetikzlibrary{arrows,positioning}
\usetikzlibrary{shadows}
\usetikzlibrary{calc}
\usetikzlibrary{shapes,snakes}
\usetikzlibrary{decorations}
\usetikzlibrary{decorations.pathreplacing}
\usepackage{hyperref}
\allowdisplaybreaks

\newlength{\bibitemsep}
\newlength{\bibparskip}
\let\oldthebibliography\thebibliography
\renewcommand\thebibliography[1]{%
  \oldthebibliography{#1}%
  \setlength{\parskip}{\bibitemsep}%
  \setlength{\itemsep}{\bibparskip}%
}

\newtheorem{theorem}{Theorem}[section]
\newtheorem{proposition}[theorem]{Proposition}

\newtheorem{lemma}[theorem]{Lemma}

\newtheorem{conjecture}[theorem]{Conjecture}
\numberwithin{equation}{section}
\theoremstyle{definition}
\newtheorem{definition}[theorem]{Definition}
\newtheorem{example}[theorem]{Example}
\newtheorem{remark}[theorem]{Remark}
\newtheorem{defprop}[theorem]{Definition/Proposition}

\newcommand{\h}{\mathfrak{h}}

\newcommand{\Q}{\mathbb{Q}}
\newcommand{\Z}{\mathbb{Z}}
\newcommand{\R}{\mathbb{R}}
\newcommand{\N}{\mathbb{N}}
\newcommand{\eN}{\overline{\mathbb{N}}}

\newcommand{\U}{\operatorname{U}}
\newcommand{\V}{\operatorname{V}}

\newcommand{\wt}{\operatorname{wt}}
\newcommand{\depth}{\operatorname{depth}}

\newcommand{\SL}{\operatorname{SL}}

\newcommand{\oko}{\zeta_q^{\mathrm{Oko}}}
\newcommand{\qmzvok}{\mathcal{Z}_q^{\mathrm{Oko}}}

\newcommand{\bone}{{\mathbf{1}}}

\DeclareRobustCommand{\gbi}{\mathfrak{g}\genfrac{(}{)}{0pt}{}}

\DeclareRobustCommand{\gsz}{\mathfrak{s}\genfrac{(}{)}{0pt}{}}

\DeclareRobustCommand{\bi}{\genfrac{(}{)}{0pt}{}}

\newcommand{\Fgh}{F_{\mathbf{h},\mathbf{g}}^\mathbf{f}}
\newcommand{\Hgh}{H_{\mathbf{h},\mathbf{g}}^\mathbf{f}}

\newcommand{\gok}{\mathfrak{o}}

\newcommand{\qsz}{\ast_\mathrm{SZ}}

\newcommand{\shsz}{\shuffle_\mathrm{SZ}}

\newcommand{\bz}{\zeta_q^\mathrm{BZ}} 
\newcommand{\oo}{\overline{1}}
\newcommand{\g}{\operatorname{g}} 
\newcommand{\sz}{\zeta_q^\mathrm{SZ}} 
\newcommand{\szs}{\zeta_q^{\mathrm{SZ},\star}}
\newcommand{\ooz}{\zeta_q^\mathrm{OOZ}}
\newcommand{\tbz}{\zeta_q^{\mathrm{TBZ}}}
\newcommand{\Fil}{\operatorname{Fil}}
\newcommand{\gr}{\operatorname{gr}}
\newcommand{\Ubz}{\operatorname{U}_\mathrm{TBZ}}
\newcommand{\Vbz}{\operatorname{V}_\mathrm{TBZ}}

\newcommand{\sooz}{\shuffle_{\mathrm{OOZ}}}

\newenvironment{itemize*}%
  {\begin{itemize}[topsep=4pt]%
    \setlength{\itemsep}{0pt}%
    \setlength{\parskip}{0pt}}%
  {\end{itemize}}
  
\newenvironment{enumerate*}%
  {\begin{enumerate}[topsep=4pt,wide,labelwidth=!,labelindent=0pt]%
    \setlength{\itemsep}{3pt}%
    \setlength{\parskip}{0pt}}%
  {\end{enumerate}}

\definecolor{darkblue}{RGB}{0,0,130}
\definecolor{mygray}{RGB}{160,160,160}
\definecolor{mycyan}{RGB}{150,255,255}
\definecolor{darkred}{RGB}{130,0,0}

\tikzset{
    >=stealth',
    punkt/.style={
           rectangle,
           rounded corners,
           draw=black, very thick,
           text width=6.5em,
           minimum height=2em,
           text centered},
    pil/.style={
           ->,
           thick,
           shorten <=2pt,
           shorten >=2pt,}
}

\makeatletter
\newcommand{\gettikzxy}[3]{%
  \tikz@scan@one@point\pgfutil@firstofone#1\relax
  \edef#2{\the\pgf@x}%
  \edef#3{\the\pgf@y}%
}
\makeatother

\title{A unified approach to $q$MZVs}
\author{Benjamin Brindle \thanks{The author has received funding from the European Research Council (ERC) under the European Union’s Horizon 2020 research and innovation programme (grant agreement No. 101001179).}}

\date{\today}
\AtEndDocument{\bigskip{\footnotesize%
  \textsc{University of Cologne, Department of Mathematics and Computer Sciences} \par  
  Weyertal 86-90, 50931 Cologne, Germany
  \par
  \textit{E-mail address:} \texttt{bbrindle@uni-koeln.de} \par
  \addvspace{\medskipamount}
}}

\begin{document}
\maketitle
\textsc{Abstract.} This survey gives a self-contained introduction to $q$-analogues of multiple zeta values ($q$MZVs). For this, we consider most common models of $q$MZVs in a unified setup going back to Bachmann and K\"uhn, such as a related quasi-shuffle product each. Also, we give distinguished translations between several models. As another unified approach to $q$MZVs, we introduce the concept of marked partitions.
\vspace{0.2cm}
\\ \vspace{-0.4cm}

\section{Introduction}

The \emph{multiple zeta value} of an \emph{admissible index} $\textbf{k} = (k_1,\dots,k_r)\in\N^r$, i.e., $k_1\geq 2,\, k_i\geq 1$ is
\begin{align*}
    \zeta(\textbf{k}) := \zeta(k_1,\dots,k_r) := \sum\limits_{m_1 > \dots > m_r > 0} \frac{1}{m_1^{k_1}}\dots \frac{1}{m_r^{k_r}},
\end{align*}
where we set $\zeta(\emptyset) := 1$. We say that $\wt(\textbf{k}) := k_1 + \dots + k_r \text{ is the \emph{weight} and}$ $\depth(\textbf{k}) := r \text{ is the \emph{depth} of \textbf{k}.}$ For well-definedness see e.g. \cite[Prop. 1.4]{Ba6}.

MZVs have representation as iterated (\emph{Kontsevich}) integrals:
\begin{align*}
    \zeta(k_1,\dots,k_r) = \int\limits_{1>t_1>\dots>t_k>0} \omega_1(t_1)\dots\omega_k(t_k),
\end{align*}
where $k:=k_1+\dots +k_r$ and
\begin{align*}
    \omega_i(t) := 
    \begin{cases}
    \frac{dt}{1-t},\ &\text{if } i\in\{k_1,k_1+k_2,\dots,k_1+\dots +k_r\},
    \\
    \frac{dt}{t},\ &\text{else.}
    \end{cases}
\end{align*}

The product of MZVs is a rational weighted sum of MZVs again. There are two important representations, the one coming from the iterated sums (so-called \emph{stuffle product}), the other from iterated integrals (\emph{shuffle product}). In particular, we get $\Q$-linear relations among MZVs, so-called double shuffle relations. After some regularization, conjecturally, these relations generate all $\Q$-linear relations among MZVs, which is why stuffle and shuffle product are of great interest.

For a more algebraic understanding, we need \emph{quasi-shuffle algebras} (introduced by Hoffman, \cite{Hof}). These are algebras, where the product is in general a $\Q$-bilinear map $\ast_\diamond$ on $\Q\left\langle A\right\rangle$ - $A$ a set, $\diamond$ an associative and commutative product on $\Q A$ - such that
\begin{align*}
    &\textbf{1}\ast_\diamond w = w\ast_\diamond \textbf{1} := w,
    \\
    &a u\ast_\diamond b v := a (u\ast_\diamond b v) + b (a u \ast_\diamond v) + (a\diamond b) (u\ast_\diamond v)
\end{align*}
for any $a,b\in\mathbb{Q}A$ and $u,v,w\in\mathbb{Q}\left\langle A\right\rangle$. The elements in $A$ are called \emph{letters}, monoids in $\Q\left\langle A\right\rangle$ \emph{words}.

Defining the free non-commutative algebra of two letters, $\mathfrak{h} := \Q\left\langle x_0,x_1\right\rangle$ ($\bone$ the empty word) and its subalgebra $\mathfrak{h}^0 := \mathbb{Q}\textbf{1}\oplus x_0\mathfrak{h}x_1$, then we can view $\zeta$ as evaluation map $\mathfrak{h}^0\rightarrow \R$ via $\bone\mapsto 1$,
\begin{align*}
    z_{k_1}\dots z_{k_r}\longmapsto \zeta(k_1,\dots,k_r)
\end{align*}
and $\Q$-linear continuation with the abbreviation $z_k := x_0^{k-1} x_1$.

On $\mathfrak{h}^0$ we consider two particular quasi-shuffle products: First, consider on $\Q A$ with
$A = \{ z_k=x_0^{k-1}x_1\, |\, k\geq 1\}$ the diamond product $z_m \diamond z_n:= z_{m+n}$. Then the induced quasi-shuffle product $\ast$ on $\Q\left\langle A\right\rangle$ is called \emph{stuffle product} (see e.g. {\cite[Def. 2.11]{Ba6}}). It is closed under restriction to $\mathfrak{h}^0$.

Second, we consider the \emph{shuffle product} (see e.g. {\cite[Def. 2.8]{Ba6}}): For $A=\{x_0,x_1\}$, i.e., $\mathbb{Q}\left\langle A\right\rangle = \mathfrak{h}$, and the diamond constant $0$, the induced quasi-shuffle product $\shuffle$ on $\mathfrak{h}$ is called the \emph{shuffle product}, which also closed under restriction to $\mathfrak{h}^0$.

They are important to us since both, $\zeta: (\mathfrak{h}^0,\ast)\rightarrow (\R,\cdot)$ and $\zeta:(\mathfrak{h}^0,\shuffle)\rightarrow (\R,\cdot)$, are algebra homomorphisms (\cite[Prop. 1]{IKZ}), the first embodying the product of MZVs obtained from iterated sums, the second the one obtained from iterated integrals.

With the integral representation, one proves the duality relations, which are distinguished linear relations among MZVs. One famous example is the equality $\zeta(2,1) = \zeta(3)$. Conjecturally, they should be implied by the (extended) double shuffle relations, but a proof for this is not known so far. 
\begin{theorem}[Duality, {\cite[§9]{Za1}}]
For every $w\in\mathfrak{h}^0$ we have
\begin{align*}
    (\zeta\circ\tau)(w) = \zeta(w),
\end{align*}
where $\tau$ is the anti-automorphism on $\mathfrak{h}$ defined via $\tau(x_i) := x_{1-i}$ for $i\in\{0,1\}$.
\end{theorem}

In the following and main section of this survey, we give a unified approach to $q$-analogues of multiple zeta values. They are important in the theory of MZVs on the one hand because of their algebraic structure and, on the other hand, since they can be viewed as holomorphic functions in the unit disc giving connections to quasi-modular forms. There exist several models of $q$MZVs, which are often introduced in different ways. One of the goals of the survey is to unify them in the following sense. We introduce \enquote{general} $q$MZVs and we will see that every model we consider is just of specific shape of these general $q$MZVs.

\begin{definition}[$q\mathrm{MZV}$, \cite{BK2}]
\label{def Z_q2}
\begin{enumerate*}
\item
Define for $r\geq 0$, $k_1,\dots,k_r\geq 0$ and polynomials $Q_1\in X\Q[X]$, $Q_2,\dots,Q_r\in\Q[X]$ with $\deg (Q_j)\leq k_j$ for all $j$
\begin{align*}
    \zeta_q(k_1,\dots,k_r;Q_1,\dots,Q_r) := \sum\limits_{m_1>\dots>m_r>0} \frac{Q_1(q^{m_1})}{(1-q^{m_1})^{k_1}}\cdots \frac{Q_r(q^{m_r})}{(1-q^{m_r})^{k_r}}\in\Q\llbracket q\rrbracket,
\end{align*}
with $\zeta_q(\emptyset,\emptyset) := 1$, where $q$ is a formal variable.

\item Furthermore, we define the space
\begin{align*}
    \hspace{-0.5cm}\mathcal{Z}_q := \left\langle \zeta_q(k_1,\dots,k_r;Q_1,\dots,Q_r)\, |\, r\geq 0,\, k_1,\dots,k_r\geq 0,\, Q_1\in X\Q[X],\, \deg (Q_j)\leq k_j\right\rangle_\Q.
\end{align*}
\end{enumerate*}
\end{definition}

Furthermore, for every model we describe the subspace of $\mathcal{Z}_q$ that it spans and give a shuffle product that it satisfies. Since it is sometimes useful (like for proving that Schlesinger--Zudilin duality and the partitions relation are equivalent; see \cite[Thm. 3.22]{Bri}), we give distinguished translations of several models into others.

These subspaces then are considered in more detail in Section 3, where we refer to the picture at the end for a brief overview.

Section 4 yields a combinatorial view to $q$MZVs. Every $q$MZV is a $q$-series in particular, and the coefficient of $q^N$ in a $q$MZV is a natural multiple of distinguished partitions of $N$. Interesting is now that for lots of models - like the Schlesinger--Zudilin model, Bradley--Zhao model, bi-brackets - these multiplicities can be interpreted as coloring of certain rows or columns in the corresponding Young diagram. This leads to the new concept of \emph{marked partitions} and has particularly application for visulization of certain theorems. For example, Schlesinger--Zudilin duality is obtained by just transposing the Young diagram inclusive the markings. Essentially for the whole section is the following:

\begin{theorem}[{\cite[Thm. 4.7.]{Bri}}]
\label{Zq poly descr0}
A $q$-series $S$ is in $\mathcal{Z}_q$ iff there exists $f=(f_r)_{r\geq 0}$ with\\ $f_r\in \Q[X_1,\dots,X_r,Y_1,\dots,Y_r]$ for $r\geq 1$ and $f_0\in\Q$ such that
\begin{enumerate*}
    \item $f_r\equiv 0$ for all but finite many $r$,
    \item  $
    S = f_0 + \sum\limits_{N\geq 1} \left(\sum\limits_{r\geq 1} \sum\limits_{((\textbf{m},\textbf{n}))\in\mathcal{P}_r(N)} f_r(m_1,\dots,m_r,n_1,\dots,n_r)\right) q^N.$
\end{enumerate*}
\end{theorem}

\section*{Acknowledgements}
I would like to thank Kathrin Bringmann for her helpful comments on this paper. Furthermore, I thank Henrik Bachmann and Ulf K\"uhn for fruitful discussions and comments while supervising my master thesis, of which this paper is part of.

\section{Models of $q$MZVs}
On the one hand, $q$-analogues of multiple zeta values are important in the theory of MZVs because of their algebraic structure and, on the other hand, since they can be viewed as holomorphic functions in the unit disc giving connections to quasi-modular forms. 

We begin with considering general modified $q$-analogues of MZVs (Def. \ref{def Z_q2}), such as their well-definedness and connection to MZVs. Furthermore, we get here a first time in contact with $\mathcal{Z}_q$, the $\Q$-algebra of $q$MZVs. A natural question is which elements generate this algebra, which leads to different models of modified $q$MZVs. Every model of $q$MZVs contains at least one algebraic aspect of MZVs: Schlesinger--Zudilin's model inherits the stuffle product, and Bradley--Zhao's model the duality of MZVs. Also interesting is Bachmann's model since it gives a deep connection to modular forms that play an important role in the theory of MZVs as already Gangl, Kaneko, and Zagier (in \cite{GKZ}) and Broadhurst and Kreimer (in \cite{BK}) have shown. For more details about the various models, we refer to the original works \cite{Bra}, \cite{Zh1}, \cite{Sch}, \cite{Zu2}, \cite{Ba2}, \cite{Tak}, \cite{OOZ}, \cite{Oko}, such as to \cite{Zh2}, where the author gives an overview of the models and their history.

In general, a $q$-analogue of an object is a modified object in an additional variable $q$ (often a series in (complex) $q$ with $|q|<1$) that returns the original object in the limit $q\rightarrow 1$, taken on the real axis from the left. 

For example, a $q$-analogue of a natural number $n$ is
\begin{align*}
    [n]_q := \frac{1-q^{n}}{1-q} = 1 + q + q^2 + \dots + q^{n-1}.
\end{align*}

\textit{Modified} $q$-analogues of MZVs are $q$-series that return a multiple zeta value if we multiply the $q$-series first with a power of $(1-q)$ (most times, $(1-q)^{\wt(\textbf{k})}$) and then take the limit $q\rightarrow 1$. This is convenient since we can avoid the additional power-of-$(1-q)$-factors without losing the structure we want to consider; furthermore, the spaces spanned by these objects become $\Q$-algebras as we will see, while the spaces spanned by non-modified objects become $\Q (1-q)$-algebras. We consider only modified $q$MZVs, why we avoid the word \enquote{modified} in the following.

\begin{remark}[to Def. \ref{def Z_q2}]
\begin{enumerate*}
\item The condition $Q_1\in X\Q[X]$ is necessary for well-definedness.

\item Note that $\mathcal{Z}_q$ does not contain all modified $q$-analogues of MZVs inevitably. For example, it is not clear yet whether the modified $q$MZVs introduced by Shen and Qin (\cite{SQ}) are in $\mathcal{Z}_q$.
\end{enumerate*}
\end{remark}

We use notation from \cite{BK2}, where the authors introduce important subspaces of $\mathcal{Z}_q$:

\begin{definition} 
Define for $d\geq 0$
\begin{align*}
    \mathcal{Z}_{q,d} :=& \left\langle \zeta_q(k_1,\dots,k_r;Q_1,\dots,Q_r)\in\mathcal{Z}_q|\, r\geq 0,\, k_1,\dots,k_r\geq 1,\deg(Q_j)\leq k_j-d\right\rangle_\Q,
    \\
    \mathcal{Z}^\circ_{q,d} :=& \left\langle \zeta_q(k_1,\dots,k_r;Q_1,\dots,Q_r)\in\mathcal{Z}_{q,d}|\, r\geq 0,\, k_1,\dots,k_r\geq 1,Q_j\in X\Q[X]\right\rangle_\Q
\end{align*}
with the abbreviation $\mathcal{Z}^\circ_q := \mathcal{Z}^\circ_{q,0}$.
\end{definition}

\begin{remark}
Naively, we could think of $k_1+\dots+k_r$ as \enquote{weight} of $\zeta_q(k_1,\dots,k_r;Q_1,\dots,Q_r)$ in accordance with the definition of weight for MZVs. But this is not well-defined since for example
\begin{align*}
    \zeta_q(k_1,\dots,k_r;Q_1,\dots,Q_r) = \zeta_q(k_1+1,\dots,k_r;(1-X)Q_1,Q_2\dots,Q_r).
\end{align*}
Hence, we need another notion of weight. Such one we will consider e.g. for bi-brackets (Def. \ref{bibr defi2}).
\end{remark}

\begin{remark}
\label{qddq rem}
\begin{enumerate*}
\item The spaces $\mathcal{Z}_q$ and $\mathcal{Z}_q^\circ$ are closed under the operator $q\frac{d}{d q}$ (\cite{Ba4}, \cite[Prop. 3.14]{BK1}).
\item The spaces $\mathcal{Z}_{q,1}$ and $\mathcal{Z}_{q,1}^\circ$ are only conjecturally closed under $q\frac{d}{d q}$ (\cite{Ba7}, \cite[Conj. 1]{Oko}). 
\end{enumerate*}
\end{remark}

\begin{proposition}[{\cite[Prop. A.32]{Bri}}]\
\label{qmzv welldef}
\begin{enumerate*}
    \item For $q\in\mathbb{C},\, |q|<1$, every $q$MZV converges and can be viewed as holomorphic function inside the unit disk or upper half plane if $q = e^{2\pi i \tau}$ with $\tau\in\mathbb{H}$.
    \item For $k_1\geq 2,\, k_2,\dots,k_r\geq 1$, the limit
    \begin{align*}
        \lim\limits_{q\rightarrow 1^-} (1-q)^{k_1+\dots + k_r}\zeta_q(k_1,\dots,k_r;Q_1,\dots,Q_r)=\zeta_q(k_1,\dots,k_r) \prod\limits_{j=1}^r Q_j(1),
    \end{align*}
    exists and is obtained by interchanging limit and summation due to absolute convergence on $[0,1)$.
\end{enumerate*}
\end{proposition}
\begin{proof}  For (i) we show the convergence of
\begin{align*}
    (1-q)^{k_1+\dots + k_r}\zeta_q(k_1,\dots,k_r;Q_1,\dots,Q_r)
\end{align*}
for all $k_1\geq 1,\, k_2,\dots,k_r\geq 0$ and for $|q| < 1$ which is equivalent to (i) because of $q\neq 1$.

Therefore, fix some $q$ with $|q|<1$, $r\geq 1$ (for $r=0$ there is nothing to do), $k_1\geq 1,\, k_2,\dots,k_r\geq 0$ and polynomials $Q_1,\dots,Q_r$ in accordance to the definition of $q$MZVs. 

We can write every $Q_j$ in the form
\begin{align*}
    Q_j = Q_j^+ - Q_j^-,
\end{align*}
where $Q_j^+$ and $Q_j^-$ are polynomials having coefficients $\geq 0$ and the same properties as $Q_j$ (i.e., $\deg\leq k_j,\, Q_1^+,Q_1^-\in X\Q[X]$). Thus, $\zeta_q(k_1,\dots,k_r;Q_1,\dots,Q_r)$ is a finite rational linear combination of $q$MZVs where the corresponding polynomials have non-negative coefficients, which is the reason why we can assume in the following w.l.o.g. that the $Q_j$ all have non-negative coefficients.

Then we have
\begin{align*}
    0\leq |Q_j(q^m)| \leq Q_j(0)+Q_j(1)|q|^m
\end{align*}
for all $|q| < 1$ and $m\in\N$.

For $j\neq 1$, we use $|q|<1$ and that the coefficients of $Q_j$ are all non-negative which gives us furthermore
\begin{align*}
    0\leq |Q_j(q^m)|\leq Q_j(0) + Q_j(1) =: c_j.
\end{align*}
Since $Q_1(0)=0$, we have for all $m\in\N$,
\begin{align*}
    0\leq |Q_1(q^m)| \leq Q_1(1) |q|^m.
\end{align*}

Now, since  
\begin{align*} \Bigg|\frac{(1-q)^k}{(1-q^m)^k}\Bigg| =\Bigg| \frac{1}{(1+q+\dots +q^{m-1})^k}\Bigg|\leq C_{q,k}
\end{align*}
for all $m$ with an appropriate constant $C_{q,m}\in\R$ (note that $\lim\limits_{m\rightarrow\infty} \Big|\frac{(1-q)^k}{(1-q^m)^k}\Big| = |(1-q)|^k$ exists since $|q|<1$), we get for all tuples $(m_1,\dots,m_r)$ of natural numbers
\begin{align*}
    0\leq |1-q|^{k_1+\dots+k_r}\prod\limits_{j=1}^r \frac{|Q_j(q^{m_j})|}{|1-q^{m_j}|^{k_j}}\leq C |q|^{m_1}
\end{align*}
with $C := Q_1(1)c_2\cdots c_r C_{q,k_1}\cdots C_{q,k_r}$.

For fixed $0\leq s < 1$, the map $x\mapsto s^x$ decreases with exponential growth in $x$. Hence, for all $n\in\N$, there is a constant $M_n$ (depends on $|q|$) such that
\begin{align*}
    |q|^{m}\leq M_n\frac{1}{m^n}
\end{align*}
for all $m\in\N$. In particular, for $n = r+1$, we obtain
\begin{align*}
    0\leq \, &|1-q|^{k_1+\dots+k_r}|\zeta_q(k_1,\dots,k_r;Q_1,\dots,Q_r)| \leq C M_{r+1} \sum\limits_{m_1>\dots>m_r>0} \frac{1}{m_1^{r+1}} 
    \\
    =\, &C M_{r+1}\sum\limits_{m_1>0} \binom{m_1 - 1}{r-1}\frac{1}{m_1^{r+1}}\leq C  M_{r+1} \sum\limits_{m>0} \frac{1}{m^2} < \infty
\end{align*}
because of $\binom{m-1}{r-1} = \frac{(m-1) (m-r+1)}{(r-1)!}\leq m^{r-1}$ and the well-known fact that the Basler sum converges. 
\\
In particular, the defining sum of $(1-q)^{k_1+\dots+k_r}\zeta_q(k_1,\dots,k_r;Q_1,\dots,Q_r)$ converges for every $|q| < 1$, which proves (i).

For the proof of (ii), we remark that on the intervall $[0,1)$ our sum converges absolutely since there, we can choose $C_{q,k} = 1$ for all $q$ and $k\geq 1$, i.e., independent of $q$. Due to this uniform convergence, limit and summation can be interchanged. Now, for all $k\geq 0,\, m\in\N$ and $1\leq j\leq r$,
\begin{align*}
    \lim\limits_{q\rightarrow 1} Q_j(q^m) = Q_j(1),\quad \lim\limits_{q\rightarrow 1} \frac{(1-q)^k}{(1-q^m)^k} = \lim\limits_{q\rightarrow 1} \left(\frac{1}{1+q+\dots+q^{m-1}}\right)^k = \frac{1}{m^k},
\end{align*}
which is why we get for fixed $k_1\geq 2,\, k_2,\dots,k_r\geq 1$ and with (i)
\begin{align*}
    \zeta_q(k_1,\dots,k_r) \prod\limits_{j=1}^r Q_j(1) &= \sum\limits_{m_1>\dots>m_r>0}  \lim\limits_{q\rightarrow 1} \prod\limits_{j=1}^r  \frac{(1-q)^{k_j} Q_j(q^{m_j})}{(1-q^{m_j})^{k_j}} 
    \\
    &= \lim\limits_{q\rightarrow 1} (1-q)^{k_1+\dots+k_r}\zeta_q(k_1,\dots,k_r;Q_1,\dots,Q_r).\tag*{\qedhere}
\end{align*}
\end{proof} 

A similar proof for the proposition with stronger assumption for (i) is given in \cite[Lem. 6.6]{BK1}.

Often, we talk about $\mathcal{Z}_q$ as algebra. We justify this in the following as we see that we can give $\mathcal{Z}_q$ a structure such that it becomes a quasi-shuffle algebra.

\begin{definition}
Consider the alphabet 
$
A_Z:=\{{\textstyle{\binom{k}{Q}}}\, |\,  Q\in \Q[X],\, k\in\N,\, \deg(Q)\leq k\}.
$
Define on $\Q A_Z$ the commutative and associative product $\diamond: \Q A_Z\otimes \Q A_Z\rightarrow \Q A_Z$ by $\Q$-bilinearity, $w\diamond \bone := \bone\diamond w := w$ for all $w\in \Q A_Z$ and
\begin{align*}
    \left(\binom{k_1}{Q_1}, \binom{k_2}{Q_2}\right) \longmapsto\binom{k_1 + k_2}{Q_1\cdot Q_2}.
\end{align*}
\end{definition}
Set $M:= A_Z^\circ\Q\left\langle A_Z\right\rangle + \Q$, where $A_Z^\circ := \{{\textstyle{\binom{k}{Q}}}: Q\in X\Q[X],\, k\in\N,\, \deg(Q)\leq k\}$, and let $\ast$ be the induced quasi-shuffle product on $M$. It turns out that $\zeta_q$ becomes an algebra homomorphism:
\begin{proposition}
\label{Zq qshuffle}
The evaluation map $\zeta_q : M\rightarrow \mathcal{Z}_q$, defined via
\begin{align*}
    \binom{k_1}{Q_1}\cdots\binom{k_r}{Q_r}\longmapsto \zeta_q(k_1,\dots,k_r;Q_1,\dots,Q_r),
\end{align*}
extended to $M$ by $\Q$-linearity, is an algebra homomorphism, i.e., for all $u,v\in M$ we have
\begin{align*}
    \zeta_q(u\ast v) = \zeta_q(u)\zeta_q(v).
\end{align*}
\end{proposition}
\begin{proof}
This follows from multiplication of $q$MZVs, represented as iterated sums.
\end{proof}

\begin{remark}[\cite{BK2}]
The subspaces $\mathcal{Z}_{q,d}, \mathcal{Z}_{q,d}^\circ\subseteq\mathcal{Z}_q$ are for all $d\geq 0$ subalgebras of $\mathcal{Z}_q$ by restricting the above defined quasi-shuffle product to the responding subspace.
\end{remark}

We need for the definition of quasi-shuffle products in different models of $q$MZVs the notion of some particular free non-commutative algebras:

\begin{definition}
Define two free non-commutative algebras of two letters,
\begin{align*}
    \mathfrak{h} :=\,  \mathbb{Q}\left\langle x_0,x_1\right\rangle,\quad \mathfrak{K} :=\,  \Q\left\langle p,y\right\rangle
\end{align*}
and subalgebras
\begin{align*}
    \mathfrak{h}^0 :=\,  \mathbb{Q}\textbf{1}\oplus x_0\mathfrak{h}x_1,\quad \mathfrak{h}^1 :=\,  \mathbb{Q}\textbf{1}\oplus\mathfrak{h}x_1,\quad 
    \mathfrak{K}^1 :=\,  \Q\bone\oplus p\mathfrak{K}y,\quad \mathfrak{K}^3 := p\Q\left\langle p,py\right\rangle py\oplus\Q\bone.
\end{align*}
Monomials in the two \enquote{letters} $x_0,\, x_1$ resp. $p,\, y$ are called \enquote{words}. They form a $\mathbb{Q}$-basis of $\mathfrak{h}$ resp. $\mathfrak{K}$. $\textbf{1}$ is the empty word and hence the unit of $\mathfrak{h}$ resp. $\mathfrak{K}$.
\end{definition}

\subsection{Schlesinger--Zudilin model}
\label{ssec: app SZ}
One of the most natural questions is how bases or at least generating systems of the $\Q$-vector space $\mathcal{Z}_q$ look like and whether there are interesting subspaces we should consider.

The probably most natural looking generating system when writing the elements of $\mathcal{Z}_q$ in the shape of Definition \ref{def Z_q2} is
\begin{align*} 
\left\{\zeta_q\left(k_1,\dots,k_r;X^{k_1},\dots,X^{k_r}\right)\, \Big|\, r\geq 0,\, k_1\geq 1,\, k_2,\dots,k_r\geq 0\right\}
\end{align*}
since it corresponds to $Q_i(X) := X^{k_i}$. We will see in Proposition \ref{Qsz=Zq2} that these indeed generate $\mathcal{Z}_q$.

These generators are named \emph{Schlesinger--Zudilin $q$MZVs} by Schlesinger (2001, \cite{Sch}) and Zudilin (2003, \cite{Zu1}) who introduced them independently:

\begin{definition}[Schlesinger--Zudilin $q$MZVs]
\label{SZ def2}\
\begin{enumerate*} 
\item An index $\mathbf{k} = (k_1,\dots,k_r)\in\N_0^r$ is \emph{SZ-admissible} if $r\geq 0$ and $\mathbf{k} = \emptyset$ or $k_1\geq 1$. 
\item Define for every SZ-admissible index $\mathbf{k}$ the SZ-$q$MZV as $\sz(\emptyset) := 1$ and for $r\geq 1$
\begin{align*}
        \sz(\mathbf{k}) := \sz(k_1,\dots,k_r) := \zeta_q\left(k_1,\dots,k_r;X^{k_1},\dots,X^{k_r}\right)
= \sum\limits_{m_1>\dots>m_r>0} \frac{q^{m_1k_1}}{(1-q^{m_1})^{k_1}} \dots \frac{q^{m_r k_r}}{(1-q^{m_r})^{k_r}}.
\end{align*}
\end{enumerate*}
\end{definition}

We introduced an extended version due to Ebrahimi-Fard, Manchon, and Singer (cf. \cite{EMS}). In the original model, due to Schlesinger and Zudilin, only indices with $k_i\geq 1$ were allowed.

\begin{remark}
\label{sz remark2}
\begin{enumerate*}
    \item Note that if one of the indices is 0 in an SZ-$q$MZV, say $k_j = 0$ for some $j$, then the summand is independent of $m_j$. Therefore, it is often useful to distinguish between zero and non-zero indices.
    
    \item An index ${\mathbf{k}}$ is SZ-admissible iff ${\mathbf{k}} + {\mathbf{1}}$ (every argument of ${\mathbf{k}}$ is increased by $1$) is admissible.
    
    \item The name of the SZ-model is attributed not only to Schlesinger, although his publication (\cite{Sch}) was two years before Zudilin's (\cite{Zu1}), since Schlesinger introduced his model in a slightly modified way: He considered
    \begin{align*}
        \zeta_q^{\mathrm{SZ'}}(k_1,\dots,k_r) := \sum\limits_{m_1>\dots>m_r>0} \frac{1}{(1-q^{m_1})^{k_1}\dots (1-q^{m_r})^{k_r}}
    \end{align*}
    with $|q| > 1$ instead of $\sz(k_1,\dots,k_r)$ (with $|q| < 1$). The latter is now the usual definition and also Zudilin introduced it this way.
    
    On closer inspection, we see that this is almost $\sz(k_1,\dots,k_r)$. Namely, one has
    \begin{align*}
        \zeta_{q^{-1}}^{\mathrm{SZ'}}(k_1,\dots,k_r) = (-1)^{k_1 + \dots + k_r}\sz(k_1,\dots,k_r).
    \end{align*}
    Further details of the history of SZ-$q$MZVs can be found, e.g. in \cite{Zh1}.
    \item For some applications as translation/duality in the OOZ-model (Thm. \ref{dual in ooz2}), it is useful to have no strictly ordered index in the defining sum of SZ-$q$MZVs. Hence, SZ-star-$q$MZVs are defined as
    \begin{align*}
        \szs({\mathbf{k}}) := \szs(k_1,\dots,k_r) := \sum\limits_{m_1\geq \dots \geq m_r>0} \frac{q^{m_1k_1}}{(1-q^{m_1})^{k_1}} \dots \frac{q^{m_r k_r}}{(1-q^{m_r})^{k_r}}.
    \end{align*}
    As for MZVs and MZSVs, every SZ-$q$MZSV is a finite sum of SZ-$q$MZVs.
\end{enumerate*}
\end{remark}

\begin{proposition}
\label{Qsz=Zq2}
SZ-$q$MZVs span $\mathcal{Z}_q$, i.e.,
\begin{align*}
    \mathcal{Z}_{q} = \left\langle \sz(k_1,\dots,k_r)\, \middle|\, r\geq 0,\, k_1\geq 1,\, k_i\geq 0\right\rangle_\Q.
\end{align*}
\end{proposition}

\begin{proof}  The proof is obtained from the fact that every expression $\frac{X^n}{(1-X)^s}$ for $0\leq n\leq s$ is a finite $\Q$-linear combination of terms $\frac{X^k}{(1-X)^k}$ for $k\geq 0$. Specifically applies

\begin{align}
\label{special rational functions2}
    \frac{X^n}{(1-X)^s} = \sum\limits_{p=n}^s \binom{s-n}{p-n}\frac{X^p}{(1-X)^p}
\end{align}
for every $s\in\mathbb{N}_0$ and every $0\leq n \leq s$.\end{proof}
In particular, the SZ-model is closed under the operator $q\frac{d}{d q}$ by Remark \ref{qddq rem}(i).

SZ-$q$MZVs satisfy a similar stuffle product as MZVs and a similar duality relation. Both combined give the shuffle product of MZVs, which is a very nice result and application of SZ-$q$MZVs. This result was mentioned in \cite{EMS} and \cite{Sin}.

\begin{definition}[SZ-stuffle product]
\label{sz stuffle2}
\begin{enumerate*}
\item Define $u_k:= p^ky\in\mathfrak{K}$ for all $k\geq 0$.

\item Consider on $\mathfrak{K}$ the usual stuffle product, i.e., define recursively the product  $\qsz:\, \mathfrak{K}\times \mathfrak{K}\rightarrow \mathfrak{K}$ via distributivity and
\begin{enumerate*}
    \item ${\mathbf{1}}\qsz w = w \qsz {\mathbf{1}} := w$,
    \item $u_s v \qsz u_t w := u_s (v\qsz u_t v) + u_t (u_s v\qsz w) + u_{s+t} (v\qsz w)$
\end{enumerate*}
for all words $v,w\in\mathfrak{K}$ and $s,t\geq 0$.
\end{enumerate*}
\end{definition}

Remark at this point that $\mathfrak{K}^1$ is generated by the words starting in some $u_k,\ k\geq 1$ and that $\mathfrak{K}^1$ is closed under $\qsz$.

\begin{definition}
Identifying $u_{k_1}\dots u_{k_r}\in\mathfrak{K}^1$ with $(k_1,\dots,k_r)$, we define the map
\begin{align*}
\sz\, :\, \mathfrak{K}^1 \rightarrow\Q\llbracket q\rrbracket,
\qquad
    u_{k_1}\dots u_{k_r}\mapsto \sz(k_1,\dots,k_r)
\end{align*}
and extend $\sz$ to $\mathfrak{K}^1$ by $\Q$-linearly.
\end{definition}
Interesting is that $\sz$ an algebra homomorphism on $(\mathfrak{K}^1,\ast_{SZ})$:
\begin{theorem}
\label{sz satisfy sz stuffle2}
The map $\sz$ is an algebra homomorphism on $(\mathfrak{K}^1,\qsz)$, i.e., for all words $v,w\in\mathfrak{K}^1$, we have
\begin{align*}
    \sz (v)\sz (w) = \sz(v\qsz w).
\end{align*}
\end{theorem}

\begin{proof}
The statement follows directly from the definition of SZ-$q$MZVs as iterated sums.
\end{proof}

We can consider SZ-$q$MZVs also in another way on $\mathfrak{K}$ :

\begin{definition}
\label{sz shuffle prod2}
\begin{enumerate*} 
\item Define recursively the SZ-$q$shuffle product $\shsz:\, \mathfrak{K}\times \mathfrak{K}\rightarrow \mathfrak{K}$ via
\begin{enumerate*}
    \item ${\mathbf{1}}\shsz w = w \shsz {\mathbf{1}} := w$,
    \item $yu \shsz v = u\shsz yv := y(u\shsz v)$,
    \item $pu\shsz pv := p(u\shsz pv) + p(pu\shsz v) + p(u\shsz v)$,
\end{enumerate*}
distributivity and $\Q$-bilinearity for all $u,v,w\in\mathfrak{K}$.

\item Identify an SZ-admissible index ${\mathbf{k}} = (k_1,\dots,k_r)$ with the word $p^{k_1}y\dots p^{k_r}y\in \mathfrak{K}^1$. Then we can define $\sz$ as the more general map
\begin{align*}
    \sz : \mathfrak{K}^1 \longrightarrow\mathcal{Z}_q,
    \qquad
    p^{k_1}y\dots p^{k_r}y \longmapsto \sz(k_1,\dots,k_r),
\end{align*}
extended to $\mathfrak{K}^1$ by $\Q$-linearity and mapping $\bone\mapsto 1$.
\end{enumerate*}
\end{definition}

Singer proved that $\sz$ is an algebra homomorphism on $(\mathfrak{K}^1,\shuffle_{SZ})$:
\begin{theorem}[{\cite[Thm. 6.2]{Sin}}]
\label{sz_shuffle2}
The map $\sz$ is an algebra homomorphism on $(\mathfrak{K}^1,\shsz)$, i.e., for all words $u,v\in\mathfrak{K}^1$ we have: 
\begin{align*}
    \sz (u)\sz(v) = \sz(u\shsz v).\tag*{\qed}
\end{align*}
\end{theorem}

Often - as for an elegant proof of SZ-duality (Thm. \ref{sz dual 2} below) - it is helpful to consider the generating series of e.g. SZ-$q$MZVs since they contain all information about the objects in a compact written term: 

\begin{theorem}[{\cite[Thm. 2.9]{Bri}}]
\label{sz generating2}
Define for every $r\geq 1$
\begin{align*}
\gsz{X_1,\dots,X_r}{Y_1,\dots,Y_r} := \sum\limits_{\substack{k_1,\dots,k_r\geq 1\\ d_1,\dots,d_r\geq 0}} \sz\left(k_1,\{0\}^{d_1},\dots,k_r,\{0\}^{d_r}\right) X_1^{k_1-1}\dots X_r^{k_r-1} Y_1^{d_1}\dots Y_r^{d_r}.
\end{align*}
Then, for every $r\geq 1$ we have with $m_{r+1} := 0$
\begin{align*}
\gsz{X_1,\dots,X_r}{Y_1,\dots,Y_r} = \sum\limits_{\substack{m_1>\dots>m_r>0\\ n_1,\dots,n_r\geq 1}} \prod\limits_{j=1}^r (1+X_j)^{n_j-1}(1+Y_j)^{m_j-m_{j+1}-1} q^{m_j n_j}.
\end{align*}
\end{theorem}

\begin{proof} 
We need the binomial theorem and the following for all $M_1,M_2,\ell\in\N_0$
\begin{align*}
\#\{n_1,\dots,n_\ell\in\mathbb{N}\, \big|\, M_1 > n_1 > \dots > n_\ell > M_2\} = \binom{M_1-M_2-1}{\ell}.
\end{align*}
With this, we get
\begin{align*}
&\gsz{X_1,\dots,X_r}{Y_1,\dots,Y_r} = \sum\limits_{\substack{k_1,\dots,k_r\geq 1\\ d_1,\dots,d_r\geq 0}} \sz\left(k_1,\{0\}^{d_1},\dots,k_r,\{0\}^{d_r}\right) X_1^{k_1-1}\dots X_r^{k_r-1} Y_1^{d_1}\dots Y_r^{d_r}
\\
=\, & \sum\limits_{\substack{k_1,\dots,k_r\geq 1\\ d_1,\dots,d_r\geq 0}}
\sum\limits_{\substack{m_1>n_1>\dots>n_{d_1} > m_2\\ > \dots > 0}} \prod\limits_{j=1}^r \frac{q^{m_j k_j}}{(1-q^{m_j})^{k_j}} X_1^{k_1-1}\dots X_r^{k_r-1} Y_1^{d_1}\dots Y_r^{d_\ell}
\\
=\, & \sum\limits_{\substack{k_1,\dots,k_r\geq 1\\ d_1,\dots,d_r\geq 0}}
\sum\limits_{m_1>\dots> m_r>0} \prod\limits_{j=1}^r \binom{m_j-m_{j+1}-1}{d_j}\frac{q^{m_j k_j}}{(1-q^{m_j})^{k_j}} X_1^{k_1-1}\dots X_r^{k_r-1} Y_1^{d_1}\dots Y_r^{d_r}
\\
=\, & \sum\limits_{\substack{m_1>\dots>m_r> 0\\ n_1,\dots,n_r > 0}}
\prod\limits_{j=1}^r \left[\sum\limits_{k_j\geq 1,d_j\geq 0} \binom{m_j-m_{j+1}-1}{d_j}\binom{n_j-1}{k_j-1} X_j^{k_j-1}Y_j^{d_j}q^{m_j n_j}\right]
\\
=\, & \sum\limits_{\substack{m_1>\dots>m_r> 0\\ n_1,\dots,n_r > 0}} \prod\limits_{j=1}^r (1+X_j)^{n_j-1}(1+Y_j)^{m_j-m_{j+1}-1}q^{m_j n_j}.\tag*{\qedhere}
\end{align*}
\end{proof}

SZ-$q$MZVs satisfy a duality relation, similar to the one of MZVs, which is, together with some related statements, why they are interesting objects.

\begin{theorem}[SZ-Duality; Zhao {\cite[Thm. 8.3]{Zh2}}]
\label{sz dual 2}
Let be $\Tilde{\tau}:\mathfrak{K}\rightarrow\mathfrak{K}$ the anti-automorphism w.r.t. concatenation, induced by $\Tilde{\tau}(p) := y,\, \Tilde{\tau}(y) := p$. On $\mathfrak{K}^1$ we have
\begin{align*}
    \sz\circ \Tilde{\tau} = \sz.\tag*{\qed}
\end{align*}
\end{theorem}

\begin{remark}[Comparison]
The SZ-model is a very elegant model of $q$MZVs since it satisfies the $q$-analogue of the stuffle product of MZVs, and is suitable to handle because of the definition that numerator and denominator have the same polynomial degree. That was also the reason why we got a generating series we can work with. But there is more: The stuffle product for the SZ-model induces together with SZ-duality the shuffle product of MZVs (cf. \cite{Sin}, for details see \cite[Thm. 3.46]{Bri}, \cite[Thm. 3.52]{Bri}).
\end{remark}

\subsubsection{SZ-shuffle on Rota--Baxter algebras}
Most quasi-shuffle products in different models can be defined on so-called Rota--Baxter algebras such that $q$MZVs in the responding model is just a special value of iterated such operators. Exemplary, we do this here for the SZ-model. For other models, see \cite{Sin}.

\begin{definition}[\emph{Rota--Baxter operator}] Let $C$ be a ring, $\lambda \in C$, and $\mathcal{A}$ a $C$-algebra.
\\
A \emph{Rota--Baxter operator} (RBO) $R$ of weight $\lambda$ on $\mathcal{A}$ over $C$ is a $C$-module endomorphism of $\mathcal{A}$ such that
\begin{align}
\label{RBO_def}
    R(x)R(y) = R(xR(y)) + R(R(x)y) + \lambda R(xy)
\end{align}
for all $x,y\in\mathcal{A}$.
\\
Furhermore, a \emph{Rota--Baxter $\mathcal{C}$-algebra} (RBA) is a pair $(\mathcal{A},R)$ with a $C$-algebra $\mathcal{A}$ and a RBO (of some weight $\lambda$) on $\mathcal{A}$ over $C$.
\end{definition}

\begin{example}
\label{ex P_q}
Consider $t\Q\llbracket t,q\rrbracket$, the vector space of formal power series $f(t) = \sum_{n\geq 0} a_n t^n$ with $a_n\in\Q\llbracket q\rrbracket$ and $\inf\{n\in\N\, |\, a_n\neq 0\}$. It can be viewed as a $\Q\llbracket q\rrbracket$-algebra. We will denote it by $\mathcal{A}_q$. 
Then a RBO of weight 1 is the operator $P_q:\mathcal{A}_q\rightarrow\mathcal{A}_q$, defined via
\begin{align*}
    P_q[f](t) :=\, & \sum\limits_{n>0} f(q^n t).
\end{align*}
\end{example}

In the following, let $T:\mathcal{A}_q\rightarrow\mathcal{A}_q$ be the operator $t\mapsto \frac{t}{1-t}$. Then we see that every SZ-$q$MZV is just the value at $t=1$ of a concatination of Operators $P_q$ and $T$:

\begin{proposition}[{\cite[Prop. 5.2]{Sin}}]
\label{sztRBO}
For every $r\geq 1$ and $k_1\geq 1,k_2,\dots,k_r\geq 0$ we have:
\begin{align*}
    P_q^{k_1}\left[T P_q^{k_2}\left[T\dots P_q^{k_r}[T]\dots \right]\right](t)
    = \sum\limits_{m_1>\dots>m_r>0} t^{m_1}\frac{q^{m_1 k_1}}{(1-q^{m_1})^{k_1}}\dots\frac{q^{m_r k_r}}{(1-q^{m_r})^{k_r}},
\end{align*}
i.e., in particular $P_q^{k_1}\left[T P_q^{k_2}\left[T\dots P_q^{k_r}[T]\dots \right]\right](1)=\sz(k_1,\dots,k_r)$.
\end{proposition}

That $\sz$ is an algebra homomorphism on $(\mathfrak{K}^1,\shuffle_{SZ})$ follows now by the more general statement, saying that $\phi_q:\mathfrak{K}\rightarrow \mathcal{A}_q$, defined through
\begin{align*}
    p^{k_1}y\cdots p^{k_r}y\longmapsto P_q^{k_1}\left[T P_q^{k_2}\left[T\cdots P_q^{k_r}[T]\cdots \right]\right]
\end{align*}
is an algebra homomorphism. This follows from obtaining that the letter $p$ corresponds to $P_q$ and $y$ to $T$. Then (ii) of the defining property of $\shsz$ (if $y$ is in one of both sides of the product leading, then we may pull it out) corresponds to the fact that by multiplication in $\mathcal{A}_q$ we may pull out some factor (since $\mathcal{A}_q$ is commutative), especially a factor $\frac{t}{1-t}$, corresponding to $T$. Analogously, we can clarify that defining property (iii) of $\shsz$ corresponds to the RHS of the definition of an RBO of weight 1, \eqref{RBO_def}. But now, the LHS of \eqref{RBO_def} will give the desired product in $\mathcal{A}_q$, making $\phi_q$ to an algebra homomorphism.

\subsection{Bradley--Zhao model}
In depth one, the BZ-model of $q$MZVs was first considered by Kaneko, Kurokawa, and Wakayama in 2002, \cite{KKW}. The general model was then introduced by Zhao in 2003 (\cite{Zh1}) and independent of Zhao by Bradley in 2004 (\cite{Bra}), clarifying its name. BZ-$q$MZVs satisfy the same duality as MZVs which is why this model plays an important role in the context of $q$MZVs.

\begin{definition}[\emph{Bradley--Zhao-$q$MZVs}]For every admissible index $\mathbf{k} = (k_1,\dots,k_r)$ we define
\begin{align*}
    \bz(\mathbf{k}) :=&\, \bz(k_1,\dots,k_r) :=\, \zeta_q\left(k_1,\dots,k_r;X^{k_1-1},\dots,X^{k_r-1}\right)
\\
=\, & \sum\limits_{m_1>\dots>m_r>0} \frac{q^{m_1(k_1-1)}}{(1-q^{m_1})^{k_1}}\cdots \frac{q^{m_r(k_r-1)}}{(1-q^{m_r})^{k_r}}.
\end{align*}
\end{definition}
In contrast to the SZ-model, BZ-$q$MZVs span only a subspace of $\mathcal{Z}_q$:
\begin{proposition}
\label{bz1X2}
The span of the BZ-model is $\mathcal{Z}_{q,1}$,
\begin{align*}
    \mathcal{Z}_{q,1} = \left\langle \bz(k_1,\dots,k_r)\, |\, r\geq 0,\, k_1\geq 2,\, k_i\geq 1\right\rangle_\Q.
\end{align*}
\end{proposition}

\begin{proof}
Every BZ-$q$MZV is by definition an element of $\mathcal{Z}_{q,1}$. That also every element of $\mathcal{Z}_{q,1}$ can be written as rational linear combination, follows from the identity
\begin{align*}
    \frac{X^s}{(1-X)^n} = \frac{X^s}{(1-X)^{s+1}}\left(1+\frac{X}{1-X}\right)^{n-s+1}
\end{align*}
for all $0\leq s< n$ and since the RHS is a rational linear combination of expressions $\frac{X^{k-1}}{(1-X)^k}$.
\end{proof}
Note that $\mathcal{Z}_{q,1}$ is a proper subspace of $\mathcal{Z}_q$ since $\zeta_q(1;X)\in\mathcal{Z}_q$, e.g., can not be written in terms of BZ-$q$MZVs. This fact can be proven with arguments similar to \cite[Thm. 2.14 (ii)]{BK1}.

BZ-$q$MZVs satisfy a quasi-shuffle product, in analogy to the stuffle product of MZVs since it is induced by multiplication of iterated sums:
\begin{definition}
\begin{enumerate*}
    \item We can define $\bz$ also as map $\bz \, :\, \h^0\rightarrow \, \mathcal{Z}_q$, defined via
    \begin{align*}
       z_{k_1}\dots z_{k_r}\longmapsto\, \bz(k_1,\dots,k_r)
    \end{align*}
    by $\Q$-linear continuation and $\mathbf{1}\mapsto 1$.
    \item Define on $\Q\{z_k:\, k\in\N\}$ the commutative and associative product $\diamond_{BZ}$ via
\begin{align*}
    z_{k_1}\diamond_{BZ} z_{k_2} := z_{k_1 + k_2} + z_{k_1 + k_2 -1}
\end{align*}
for all $k_1,k_2\geq 1$ and $\mathbf{1}\diamond_{BZ} w := w\diamond_{BZ}\mathbf{1} := w$ for all $w\in\mathfrak{h}^0$. Let be $\ast_{BZ}$ the induced quasi-shuffle product on $\mathfrak{h}^1$. 
\end{enumerate*}
\end{definition}
Notice that $\mathfrak{h}^0\subset\mathfrak{h}^1$ is closed under $\ast_{BZ}$. We have the following:

\begin{proposition}
\label{bz qshuffle2}
On $(\mathfrak{h}^0,\ast_{BZ})$, $\bz$ is an algebra homomorphism, i.e., for all $u,v\in\mathfrak{h}^0$ we have
\begin{align*}
    \bz(u\ast_{BZ} v) = \bz(u)\bz(v).
\end{align*}
\begin{proof}  
This is elementary calculation using the product of iterated sums and the fact
\begin{align*}
    \frac{q^{m (k-2)}}{(1-q^m)^k} = \frac{q^{m (k-2)}}{(1-q^m)^{k-1}} + \frac{q^{m (k-1)}}{(1-q^m)^k}.\tag*{\qedhere}
\end{align*}
\end{proof}
\end{proposition}
Define for the following the abbreviation
\begin{align*}
    f_\ell(X,Y) := \bi{X,0,\dots,0}{Y,\dots,Y}
\end{align*}
for all $\ell\in\N$, where both rows have $\ell$ entries. Furthermore, we identify $(f_{\ell_1},\dots,f_{\ell_r})$, $\ell_j,r\in\N$, with the concatenation of $f_{\ell_1},\dots,f_{\ell_r}$.
\begin{theorem}[{\cite[Thm. 2.13]{Bri}}]
\label{bz genser2}
Define for $r\geq 1$ the generating series of BZ-$q$MZVs,
\begin{align*}
    \, \mathfrak{b}\bi{X_1,\dots,X_r}{Y_1,\dots,Y_r} 
        = \sum\limits_{\substack{k_1,\dots,k_r\geq 1\\ d_1,\dots,d_r\geq 1}} \bz\left(k_1 + 1,\{1\}^{d_1 - 1},\dots,k_r + 1,\{1\}^{d_r - 1}\right) X^{k_1 - 1}Y_1^{d_1 - 1}\cdots X_r^{k_r - 1} Y_r^{d_r - 1}.
\end{align*}
Then, we have for every $r\geq 1$
\begin{align*}
    \, \mathfrak{b}\bi{X_1,\dots,X_r}{Y_1,\dots,Y_r} 
        = \sum\limits_{\substack{\ell_1,\dots,\ell_r \geq 1\\ \delta_1,\dots,\delta_r\in\{0,1\}}} (-1)^{r-(\delta_1+\dots+\delta_r)} \mathfrak{s}\left(f_{\ell_1}(\delta_1 X_1,Y_1),\dots, f_{\ell_r}(\delta_r X_r, Y_r)\right) \prod\limits_{j=1}^r (1+\delta_j X_j) Y_j^{\ell_j-1}.
\end{align*}
\end{theorem}

\begin{proof}
We give a proof for clarity reasons only for $r=1$. The statement for general $r$ can be proven in the same way: So, we have to show that
\begin{align*}
    \mathfrak{b}\bi{X}{Y}
    =
    \sum\limits_{\ell\geq 1,\, \delta\in\{0,1\}} (-1)^{1-\delta} \mathfrak{s}\left(f_\ell(\delta X,Y)\right) (1+\delta X) Y^{\ell}.
\end{align*}
Using the definition of $\mathfrak{b}$ and $\mathfrak{s}$ as generating series of BZ- resp. SZ-$q$MZVs, this equality is equivalent to (note that we plug in $X_2=\dots=X_\ell=0$ in each summand resp. $X_1=0$ additionally in the last sum)
\begin{align*}
    &\,\sum\limits_{k,\, d\geq 1} \bz\left(k+1,\{1\}^{d-1}\right) X^{k}Y^{d}
    \\
    =&\,
    \sum\limits_{\ell\geq 1} \sum\limits_{\substack{k_1\geq 1,k_2=\dots =k_\ell= 1\\ d_1,\dots,d_\ell\geq 0}} \sz\left(k_1,\{0\}^{d_1},\dots,k_\ell,\{0\}^{d_\ell}\right) X^{k_1-1} Y^{d_1+\dots+d_\ell} (1+X)Y^{\ell}
    \\
    &\, - \sum\limits_{\ell\geq 1} \sum\limits_{\substack{k_1=\dots =k_\ell= 1\\ d_1,\dots,d_\ell\geq 0}} \sz\left(k_1,\{0\}^{d_1},\dots,k_\ell,\{0\}^{d_\ell}\right) Y^{d_1+\dots+d_\ell} Y^{\ell}.
\end{align*}
The RHS equals
\begin{align*}
    &\sum\limits_{\ell\geq 1} \sum\limits_{\substack{k_1>1\\ d_1,\dots,d_\ell\geq 0}} \sz\left(k_1,\{0\}^{d_1},1,\{0\}^{d_2},\dots,1,\{0\}^{d_\ell}\right) X^{k_1 - 1} Y^{d_1+\dots + d_\ell + \ell}
    \\
    &\, +\sum\limits_{\ell\geq 1}\sum\limits_{\substack{k_1\geq 1\\ d_1,\dots,d_\ell\geq 0}} \sz\left(k_1,\{0\}^{d_1},1,\{0\}^{d_2},\dots,1,\{0\}^{d_\ell}\right) X^{k_1}Y^{d_1+\dots +d_\ell + \ell}
    \\
    =\,&\sum\limits_{\ell\geq 1} \sum\limits_{\substack{k_1>1\\ d_1,\dots,d_\ell\geq 0}} \sz\left(k_1,\{0\}^{d_1},1,\{0\}^{d_2},\dots,1,\{0\}^{d_\ell}\right) X^{k_1 - 1} Y^{d_1+\dots + d_\ell + \ell}
    \\
    &\, +\sum\limits_{\ell\geq 1}\sum\limits_{\substack{k_1 > 1\\ d_1,\dots,d_\ell\geq 0}} \sz\left(k_1 - 1,\{0\}^{d_1},1,\{0\}^{d_2},\dots,1,\{0\}^{d_\ell}\right) X^{k_1 - 1}Y^{d_1+\dots +d_\ell + \ell}
    \\
    =\, &\sum\limits_{d\geq 0}\sum\limits_{\varepsilon_1,\dots,\varepsilon_d\in\{0,1\}}\sum\limits_{k_1>1} \sz\left(k_1,\varepsilon_1,\dots,\varepsilon_d\right)X^{k_1-1}Y^{d+1}
    \\
    &\, -\sum\limits_{d\geq 0}\sum\limits_{\varepsilon_1,\dots,\varepsilon_d\in\{0,1\}}\sum\limits_{k_1>1} \sz\left(k_1 - 1,\varepsilon_1,\dots,\varepsilon_d\right)X^{k_1-1}Y^{d+1}
    \\
    =\,& \sum\limits_{k_1>1,\, d\geq 0} \bz\left(k_1,\{1\}^d\right) X^{k_1 - 1} Y^{d+1}
    =\, \sum\limits_{k,d\geq 1} \bz\left(k+1,\{1\}^{d-1}\right) X^{k} Y^{d}.
\end{align*}
In the second last step we used an explicit translation of BZ-$q$MZVs into SZ-$q$MZVs that can be found e.g. in Proposition \ref{extBZ-to-SZ2}. Furthermore, the statement for general $r\geq 1$ follows by an analogous calculation.
\end{proof}

\begin{remark}
We can formulate the theorem also more understandable:\\ For $r\geq 1$ we have
\begin{align*}
    \mathfrak{b}\bi{X_1,\dots,X_r}{Y_1,\dots,Y_r} \equiv \sum\limits_{\substack{\ell_1,\dots,l_r \geq 1}}  \mathfrak{s}\left(f_{\ell_1}(X_1,Y_1),\dots, f_{\ell_r}(X_r, Y_r)\right) \prod\limits_{j=1}^r (1+\delta_j X_j) Y_j^{\ell_j-1}
\end{align*}
modulo terms not divisible by $\prod\limits_{j=1}^r X_j Y_j$. 
\end{remark}

One of the main reasons why BZ-$q$MZVs are of interest is that they satisfy the same duality as MZVs:
\begin{theorem}[BZ-Duality; Bradley {\cite[Thm. 5]{Bra}}, Seki--Yamamoto {\cite[Thm. 1.2]{SY}}]
\label{bz dual2}
We have on $\mathfrak{h}^0$
\begin{align*}
    \bz\circ\tau = \bz. \tag*{\qed}
\end{align*}
\end{theorem}

Besides this duality, there is a bigger class of relations, the $q$-Ohno relations. BZ-duality is a special case of them:
\begin{theorem}[$q$-Ohno relations; Okuda--Takeyama {\cite[Thm. 1]{OT}}]
\label{qOhno}
For any admissible index ${\mathbf{k}} = (k_1,\dots,k_r)$ and any $c\in\N_0$ we have
\begin{align*}
    \sum\limits_{ |{\mathbf{c}}| = c} \bz\left({\mathbf{k}} + {\mathbf{c}}\right) = \sum\limits_{ |{\mathbf{c}}| = c} \bz\left({\mathbf{k}}^\vee + {\mathbf{c}}\right),
\end{align*}
where we sum over all ${\mathbf{c}} = (c_1,\dots,c_r)\in\N_0^r$ with $|{\mathbf{c}}| := c_1+\dots + c_r = c$.
\end{theorem}

\begin{remark}[Comparison]
The indisputable advantage of this model is that it satisfies the same duality as MZVs. In particular, the latter follows from BZ-duality by taking the limit $q\rightarrow 1$ after multiplication with $(1-q)^{k_1+\dots+k_r}$.

On the other hand, it is a bit challenging to handle, like finding a \enquote{good} generating series. For these things, the SZ-model is most times superior to the BZ-model.
\end{remark}

\subsection{Bi-brackets}

Another interesting model of $q$-analogues are so-called brackets (introduced in Bachmann's master thesis \cite{Ba1}, further investigated, e.g., in \cite{BK1}) and their generalization, bi-brackets, introduced by Bachmann in his PhD thesis (\cite{Ba2}).

The motivation for introducing bi-brackets came from examining the Fourier expansion of Eisenstein series and their generalization, Multiple Eisenstein series, such as their derivatives (studied in \cite{Ba4}). From there, the original definition is justified:
\begin{defprop}[{\cite{Ba4}}]
\label{bibr defi2}
\begin{enumerate*}
\item Define for $r\geq 0,\, k_1,\dots,k_r\geq 1,\, d_1,\dots,d_r\geq 0$ the bi-bracket 
\begin{align*}
    \g\bi{k_1,\dots,k_r}{d_1,\dots,d_r} :=&\,
    \sum\limits_{\substack{m_1>\dots>m_r>0\\ n_1,\dots,n_r > 0}} \frac{m_1^{d_1}}{d_1!}\cdots \frac{m_r^{d_r}}{d_r!} \frac{n_1^{k_1 - 1}\dots n_r^{k_r -1}}{(k_1-1)!\dots (k_r - 1)!} q^{m_1 n_1 + \dots + m_r n_r}
    \\
    =&\,
    \sum\limits_{m_1>\dots>m_r>0} \frac{m_1^{d_1}}{d_1!}\dots\frac{m_r^{d_r}}{d_r!} \frac{P_{k_1}(q^{m_1})}{(1-q^{m_1})^{k_1}}\dots \frac{P_{k_r}(q^{m_r})}{(1-q^{m_r})^{k_r}},
\end{align*}
where $P_k$ is the $k$-th \emph{Eulerian polynomial}, defined via
\begin{align*}
    \frac{P_k(X)}{(1-X)^k} := \sum\limits_{n>0} \frac{n^{k-1}}{(k-1)!} X^n.
\end{align*}
Additionally, we set $\g\bi{\emptyset}{\emptyset} := 1$ as usually. Furthermore, denote by $k_1+\dots+k_r$ the \emph{weight} and by $r$ the \emph{depth} of the bi-bracket.
\item We define $\g(\emptyset) := 1$ and for any $r\geq 1$ and any index $\textbf{k} = (k_1,\dots,k_r)\in\mathbb{N}^r$ the \emph{bracket} of $\textbf{k}$ as 
\begin{align*}
    \g(\textbf{k}) :=\, \g(k_1,\dots,k_r) :=\, \zeta_q\left(k_1,\dots,k_r;P_{k_1},\dots,P_{k_r}\right)
    =\, \sum\limits_{m_1>\dots>m_r>0} \frac{P_{k_1}(q^{m_1})}{(1-q^{m_1})^{k_1}}\dots \frac{P_{k_r}(q^{m_r})}{(1-q^{m_r})^{k_r}}.
\end{align*}
\end{enumerate*}
\end{defprop}

\begin{remark}
\begin{enumerate*}
    \item An explicit expression of Eulerian polynomials is
    \begin{align*}
        P_k(X) = \frac{1}{(k-1)!}\sum\limits_{n=1}^{k} \left(\sum\limits_{j=0}^{n-1} (-1)^j\binom{k}{j} (n-j)^{k-1}\right) X^n.
    \end{align*}
    \item The name \textit{(bi-)bracket} comes from the original notation where $[\dots]$ was used instead of $\g(\dots)$.
    \item Every bracket $\g(k_1,\dots,k_r)$ is a bi-bracket since
    $
        \g(k_1,\dots,k_r) = \g\bi{k_1,\dots,k_r}{0,\dots,0}.
    $
\end{enumerate*}
\end{remark}
Bi-brackets generalize Eisenstein series since for even $k$, $\g\bi{k}{0}$ is the usual Eisenstein series of weight $k$, $G_k$, minus the constant term, i.e.,
\begin{align*}
    G_k = -\frac{B_k}{2 k!} + \g\bi{k}{0}.
\end{align*}
Furthermore, for every $d>0$, we have
\begin{align*}
    \left(q\frac{d}{dq}\right)^d G_k = \frac{(k+d-1)!d!}{(k-1)!} \g\bi{k+d}{d}.
\end{align*}
With the observation done before, one can obtain that the space of quasi-modular forms, which is $\Q[G_2,G_4,G_6]$, is a proper subspace of $\mathcal{Z}_q$. In this way, we get a connection to modular forms, which play an important role in the theory of MZVs as considered e.g. in \cite{GKZ}.

Bi-brackets and their structure are well known, for more details than in this section we refer to \cite{Ba3}, \cite{Ba4}, \cite{Ba5}, \cite{BK1}, \cite{BK2}, \cite{Zu2}.

\begin{theorem}[{\cite[Thm. 2.3]{BK2}}]
\label{bibr Zq2}
Bi-brackets span the space $\mathcal{Z}_q$, i.e.,
\begin{align*}
    \mathcal{Z}_{q} =\, \left\langle \g\bi{k_1,\dots,k_r}{d_1,\dots,d_r}\, \middle|\, r\geq 0,\, k_i\geq 1,\, d_i\geq 0\right\rangle_\Q. 
\end{align*}
\end{theorem}

Also, the algebra of bi-brackets can be viewed as a quasi-shuffle algebra: 
\begin{definition}[\cite{Ba4}] Set $A_z^{bi} := \{ z_{k,d}\, |\,  k,d\in\mathbb{N}_0,\, k\geq 1\}$
and define on $\Q A_z^{bi}$ the product $\boxcircle$ by
\begin{align*}
    z_{k_1,d_1}\boxcircle z_{k_2,d_2} := &\binom{d_1 + d_2}{d_1} \sum\limits_{1\leq j\leq k_1} \lambda_{k_1,k_2}^j z_{j,d_1 + d_2} + \binom{d_1 + d_2}{d_1} \sum\limits_{1\leq j\leq k_2} \lambda_{k_2,k_1}^j z_{j,d_1 + d_2} 
    \\
    &+ \binom{d_1 + d_2}{d_1} z_{k_1 + k_2,d_1 + d_2}
\end{align*}
and $\Q$-bilinear continuation to $\Q A_z^{bi}$. Here, $\lambda_{a,b}^j$ is defined as
\begin{align*}
    \lambda_{a,b}^j := (-1)^{b-1} \binom{a+b-j-1}{a-j}\frac{B_{a+b-j}}{(a+b-j)!}.
\end{align*}
\end{definition}

The map $\boxcircle$ is associative and commutative (\cite{Ba4}), i.e., it induces a quasi-shuffle product $\boxast$:

\begin{theorem}
\begin{enumerate*} 
\item (\cite[Thm. 3.6]{Ba4}): The evaluation map $\g : (\Q A_z^{bi},\boxast)\rightarrow (\mathcal{Z}_q,\cdot)$, defined via
\begin{align*}
    z_{k_1,d_1}\dots z_{k_r,d_r}\longmapsto \g\bi{k_1,\dots k_r}{d_1,\dots,d_r},
\end{align*}
$\g(\emptyset) := 1$ and $\Q$-linear continuation, is an algebra homomorphism.
\item (\cite[Thm. 2]{Zu2}): The quasi-shuffle product $\boxast$ implies the stuffle product of MZVs. 
\end{enumerate*}
\end{theorem}

As for SZ-$q$MZVs, also for bi-brackets it is often convenient to work with their generating series because of the quite compact written term containing all information. Bachmann did this already in his PhD thesis:

\begin{theorem}[{\cite[Theorem 2.3]{Ba4}}]
\label{g generating2}
Let be 
\begin{align*}
\gbi{X_1,\dots,X_r}{Y_1,\dots,Y_r} := \sum\limits_{\substack{k_1,\dots,k_r>0\\ d_1,\dots,d_r>0}} \g\bi{k_1,\dots,k_r}{d_1-1,\dots,d_r-1} X_1^{k_1-1}\dots X_r^{k_r-1} Y_1^{d_1-1}\dots Y_r^{d_r-1}
\end{align*}
for every $r\geq 1$. Then we have
\begin{align*}
\gbi{X_1,\dots,X_r}{Y_1,\dots,Y_r} = \sum\limits_{\substack{m_1>\dots>m_r>0\\ n_1,\dots,n_r\geq 1}} \prod\limits_{j=1}^r e^{m_j Y_j}e^{n_j X_j}q^{m_j n_j}.
\end{align*}
\end{theorem}

\begin{proof}  We have for every $r\geq 1$
\begin{align*}
    \, \gbi{X_1,\dots,X_r}{Y_1,\dots,Y_r} 
    =\, & \sum\limits_{\substack{k_1,\dots,k_r>0\\ d_1,\dots,d_r>0}} \sum\limits_{\substack{m_1>\dots > m_r > 0\\ n_1,\dots, n_r > 0}} \left(\prod\limits_{j=1}^r \frac{(m_j Y_j)^{d_j-1}}{(d_j-1)!}\right) \left(\prod\limits_{j=1}^r \frac{(n_j X_j)^{k_j-1}}{(k_j-1)!}\right) q^{m_1n_1 + \dots + m_r n_r}
    \\
    =\, & \sum\limits_{\substack{m_1>\dots > m_r > 0\\ n_1,\dots, n_r > 0}} \left(\prod\limits_{j=1}^r \sum\limits_{d_j = 1}^\infty \frac{(m_j Y_j)^{d_j - 1}}{(d_j - 1)!}\right) \left(\prod\limits_{j=1}^r \sum\limits_{k_j = 1}^\infty \frac{(n_j X_j)^{k_j - 1}}{(k_j - 1)!}\right)  q^{m_1n_1 + \dots + m_r n_r}
    \\
    =\, & \sum\limits_{\substack{m_1>\dots>m_r>0\\ n_1,\dots,n_r\geq 1}} \prod\limits_{j=1}^r e^{m_j Y_j}e^{n_j X_j}q^{m_j n_j}.\tag*{\qedhere}
\end{align*}
\end{proof} 

One application of the generating series of bi-brackets is for proving a nice identity among bi-brackets, the so-called \emph{partition relation}, which can be viewed as a kind of duality in the model of bi-brackets:

\begin{theorem}[Partition relation, {\cite[Theorem 2.3]{Ba4}}]
\label{partition rel2}
For all $r\geq 1$ we have
\begin{align}
\label{partition relation2}
    \gbi{X_1,\dots,X_r}{Y_1,\dots,Y_r} = \gbi{Y_1+\dots +Y_r,\dots,Y_1+Y_2,Y_1}{X_r,X_{r-1}-X_r,\dots,X_1-X_2}.
\end{align}
\end{theorem}

\begin{remark}
The name of this relation comes from the fact that $\gbi{X_1,\dots,X_r}{Y_1,\dots,Y_r}$ is a sum over all partitions with exactly $r$ distinct parts and the relation itself is obtained by taking the sum over the partitions with transposed Young diagram.
\end{remark}

Another application of the generating series of bi-brackets is to give elegant translations between bi-brackets and the SZ-model. That is possible since SZ-$q$MZVs, as well as bi-brackets, span $\mathcal{Z}_q$ (Prop. \ref{Qsz=Zq2}, Thm. \ref{bibr Zq2}):

\begin{theorem}[Translation bi-brackets-SZ-model, {\cite[Thm. 2.18]{Bri}}]
\label{gsz relation2}\
\begin{enumerate*}
    \item For every $r\geq 1$ we have
\begin{align*}
\prod\limits_{j=1}^r e^{X_j}e^{Y_1+\dots +Y_{j}} \cdot \gsz{e^{X_1}-1,\dots,e^{X_r}-1}{e^{Y_1}-1,\dots,e^{Y_1+\dots +Y_r}-1}
=
\gbi{X_1,\dots,X_r}{Y_1,\dots,Y_r}.
\end{align*}
\item For every $r\geq 1$ we have
\begin{align*}
\hspace{-0.5cm}\gsz{X_1,\dots,X_r}{Y_1,\dots,Y_r} 
= \left(\prod\limits_{j=1}^r (1+X_j)(1+Y_1+\dots +Y_j)\right)^{-1}  \gbi{\ln(X_1 + 1),\dots,\ln(X_r + 1)}{\ln(Y_1 + 1),\dots,\ln(Y_1 + \dots +Y_r + 1)}.
\end{align*}
\end{enumerate*}
\end{theorem}
\begin{proof}  
\begin{enumerate*}
\item We are done by multiplying both sides in Theorem \ref{sz generating2} with $\prod\limits_{j=1}^r (1+X_j)(1+Y_j)$ and substituting then
\begin{align*}
    X_j \mapsto e^{X_j} - 1,\quad Y_j \mapsto e^{Y_1 + \dots + Y_j} -1\qquad \text{for all } 1\leq j\leq r.
\end{align*}

\item We obtain the claim by substituting for all $1\leq j\leq r$ in (i)
\begin{align*}
    X_j \mapsto \ln(X_j + 1),\quad Y_j\mapsto \ln(Y_1 + \dots + Y_j + 1).\tag*{\qedhere}
\end{align*}
\end{enumerate*}
\end{proof} 

\begin{remark}
From Theorem \ref{gsz relation2} we obtain a new proof of the well-known fact that bi-brackets and SZ-$q$MZVs span the same $\Q$-vector space, $\mathcal{Z}_q$.
\end{remark}

We obtain a direct, but less elegant, translation of bi-brackets into SZ-$q$MZVs when using the identity \eqref{special rational functions2} and elementary calculations:

\begin{theorem}[{\cite[Thm. 2.19]{Bri}}]
\label{sz-g-explicit2}
For every $r\in\mathbb{N}$, $k_1,\dots,k_r\in\mathbb{N}$, $d_1,\dots,d_r\in\mathbb{N}_0$, we have:
\begin{align*}
    \g\bi{k_1,\dots,k_r}{d_1,\dots,d_r} = \sum\limits_{\substack{1\leq n_j\leq p_j \leq k_j\\ 0\leq f_j\leq d_j\\ 1\leq j\leq r}} & c_{\mathbf{k}}^\mathbf{d} (\mathbf{n})\cdot \left[\prod\limits_{j=1}^r \binom{d_j}{f_j}\binom{k_j-n_j}{p_j-n_j}\right]
    \\
    &\times \prod\limits_{j=1}^r \sum\limits_{g_j = 0}^{\Fgh (j)} \binom{\Fgh(j)}{g_j}
    \sum\limits_{\ell_j = 0}^{g_j} \left(\delta_{g_j = 0} + \sum\limits_{\substack{s_1+\dots+s_{\ell_j}=g_j\\ s_i\geq 1}}\binom{g_j}{s_1,\dots,s_{\ell_j}}\right)
    \\
    &\times\sum\limits_{h_j=0}^{\Hgh(j)}\binom{\Hgh(j)}{h_j} \times \sz\left(p_1,\{0\}^{\ell_1},\dots,p_r,\{0\}^{\ell_r}\right)
\end{align*}
with $c_{\mathbf{k}}^\mathbf{d} (\mathbf{n}) := \prod\limits_{l = 1}^r \frac{1}{d_\ell!(k_\ell-1)!} \left(\sum\limits_{i=0}^{n_\ell-1} (-1)^{i}\binom{k_\ell}{i}(n_\ell-i)^{k_\ell-1}\right)\in\mathbb{Q}$ and for all $1\leq j\leq r$,
\begin{align*}
    \Fgh(j) := f_j + \sum\limits_{i = 1}^{j-1} (f_i - g_i - h_i),\quad \Hgh(j) := \Fgh(j) - g_j = f_j - g_j + \sum\limits_{i = 1}^{j-1} (f_i - g_i - h_i).
\end{align*}
\end{theorem}

\begin{proof} According to the definition of bi-brackets and \eqref{special rational functions2}, we get:
\begin{align*}
    &\g\bi{k_1,\dots,k_r}{d_1,\dots,d_r} = 
    \sum\limits_{m_1>\dots>m_r>0} \frac{m_1^{d_1}}{d_1!}\cdot \dots \cdot \frac{m_r^{d_r}}{d_r!}\frac{P_{k_1}(q^{m_1})}{(1-q^{m_1})^{k_1}}\cdot \dots\cdot \frac{P_{k_r}(q^{m_r})}{(1-q^{m_r})^{k_r}}
    \\
    =\, & \sum\limits_{m_1>\dots>m_r>0} \prod\limits_{j=1}^r \frac{m_j^{d_j}}{d_j!}\left[\frac{1}{(k_j-1)!}\sum\limits_{n_j=1}^{k_j} \left(\sum\limits_{i=0}^{n_j-1} (-1)^i\binom{k_j}{i}(n_j-i)^{k_j-1}\right) q^{m_j n_j}\right]\frac{1}{(1-q^{m_j})^{k_j}}
    \\
    =\, &\, \sum\limits_{\substack{m_1>\dots>m_r>0\\ 1\leq n_j \leq k_j\\ 1\leq j\leq r}} c_{\mathbf{k}}^\mathbf{d} (\mathbf{n})\prod\limits_{j=1}^r m_j^{d_j}\frac{q^{m_j n_j}}{(1-q^{m_j})^{k_j}}
    \\
    =&\, \sum\limits_{\substack{1\leq n_j \leq k_j\\ 1\leq j\leq r}} c_{\mathbf{k}}^\mathbf{d} (\mathbf{n}) \sum\limits_{m_1>\dots>m_r>0} \prod\limits_{j=1}^r \left(\sum\limits_{f_j=0}^{d_j} \binom{d_j}{f_j}(m_j-1)^{f_j}\sum\limits_{p_j=n_j}^{k_j} \binom{k_j-n_j}{p_j-n_j}\frac{q^{m_j p_j}}{(1-q^{m_j})^{p_j}}\right)
    \\
    =\, &\, \sum\limits_{\substack{1\leq n_j\leq p_j \leq k_j\\ 0\leq f_j\leq d_j\\ 1\leq j\leq r}} c_{\mathbf{k}}^\mathbf{d} (\mathbf{n}) \left(\prod\limits_{j=1}^r \binom{d_j}{f_j}\binom{k_j-n_j}{p_j-n_j}\right)\sum\limits_{m_1>\dots>m_r>0}\prod\limits_{j=1}^r (m_j-1)^{f_j}\frac{q^{m_j p_j}}{(1-q^{m_j})^{p_j}}.
\end{align*}

To avoid confusion about the above big sum, we will concentrate on the last sum, which will give the stated theorem via the following combinatorial argument:

We can interpret the factor $(m_j-1)^{f_j}$ as the same sum without this factor and additional summands $b_1^j,\dots, b_{f_j}^j$ (for better understanding, the second index, $j$, is in the exponent; this is also for $a_i^j$ below the case) where they range between $0$ and $m_j$, i.e., $m_j > b_1^j,\dots, b_{f_j}^j > 0$.

Since these additional summands are not ordered, we have to make them ordered via we distinguish between 
\begin{itemize*}
    \item how many of the $b_i^j$ can have the property $m_{t} > b_i^j > m_{t+1}$ for $1\leq t\leq r-1$,
    \item how many of such $b_i^j$ can coincide, 
    \item how many of the $b_i^j$ can be equal $m_t$ for some $2\leq t\leq r$.
\end{itemize*}
This is done inductive:
\begin{itemize*}
    \item The property $m_1 > b_i^j > m_2$ can have $0$ upto $f_1$ of the $b_i^j$. For every such number, $g_1$, we have $\binom{f_1}{g_1}$ choices, which $g_1$ of the possible $b_i^j$ lie between $m_1$ and $m_2$. This gives the $\sum\limits_{g_1 = 0}^{\Fgh(1)} $-sum (note that $\Fgh(1) = f_1$).
    \item Independent of the above choice, the $g_1$ chosen $b_i^j$ can coincide. The ways, how this is possible, are: they coincide all ($l_1 = 1$), they can coincide in two different ways ($l_1= 2$) and so on or they are all different $l_1 = g_1$. The number of ways how they can coincide gives the multinomial coefficient. If $g_1 = 0$, nothing coincides ($l_1 = 0$) and there is exact one possibility, which clarifies the Kronecker-delta $\delta_{g_1 = 0}$ and so the $\sum\limits_{\ell_2 = 0}^{g_1}$-sum is declared.
    \item Now, there are $f_1-g_1$ of the $b_i^j$ of the first item left, i.e., they are all at most $m_2$. Hence, $0$ upto $f_1-g_1$ of them can equal $m_2$. For every such number, $h_1$, there are $\binom{f_1-g_1}{h_1}$ possibilities, which of them are equal $m_2$. This clarifies the $\sum\limits_{h_1=0}^{\Hgh(1)}$-sum (note that $\Hgh(1) = f_1 - g_1$).
    \item Now, $f_1-g_1-h_1$ of the $b_i^j$ of the first item are left and all are less than $m_2$.
    \item So we can do every of the above three steps (i)-(iii) again, now with $m_2$ and $m_3$ instead of $m_1$ and $m_2$ (and $\Fgh(2) = f_2+(f_1 - g_1 - h_1)$ instead of $\Fgh(1)$ for the number of $b_i^j$'s less than $m_2$).
    \item We are now done since we get as remaining part of the formula (after ordering the indices and setting $m_{r+1} := 0$ as usual) the sum
    \begin{align*}
        \sum\limits_{\substack{m_j > a_1^j > \dots > a_{\ell_j}^j > m_{j+1}\\ 1\leq j\leq r}} \frac{q^{m_1 p_1}}{(1-q^{m_1})^{p_1}}\cdots \frac{q^{m_r p_r}}{(1-q^{m_r})^{p_r}} = \sz\left(p_1,\{0\}^{\ell_1},\dots , p_r,\{0\}^{\ell_r}\right).\tag*{\qedhere}
    \end{align*}
\end{itemize*}
\end{proof}

\begin{example} 
We have, for example,
\begin{align*}
   \g\bi{3,2}{0,2}
   =\, & \frac{1}{4}\left(\sz(1,1) + \sz(1,2) + 3\sz(2,1) + 3\sz(2,2) + 2\sz(3,1) + 2\sz(3,2)\right)
   \\ &+\frac{3}{4} \left(\sz(1,1,0) + \sz(1,2,0) + 3\sz(2,1,0) + 3\sz(2,2,0) + 2\sz(3,1,0)+ 2\sz(3,2,0) \right)
   \\ &+\frac{1}{2}\big(\sz(1,1,0,0) + \sz(1,2,0,0) + 3\sz(2,1,0,0) + 3\sz(2,2,0,0)\\ &\qquad + 2\sz(3,1,0,0)  + 2\sz(3,2,0,0)\big).
\end{align*}
\end{example}

\begin{remark}
Zudilin's model of $q$MZVs, \emph{multiple $q$-zeta brackets}, is closely related to bi-brackets (\cite{Zu2} for details). They are defined for $k_1,\dots,k_r,\, d_1,\dots,d_r\geq 1$ ($r\geq 0$) as
\begin{align*}
    \mathfrak{Z}_q\bi{k_1,\dots,k_r}{d_1,\dots,d_r} := c \sum\limits_{\substack{m_1,\dots,m_r>0\\ n_1,\dots,n_r > 0}} m_1^{d_1 - 1}n_1^{k_1-1}\cdots m_r^{d_r - 1}n_r^{k_r-1} q^{(m_1+\dots+m_r)n_1+\dots+m_r n_r}
\end{align*}
with $c:=\frac{1}{\prod\limits_{j=1}^r (k_j-1)!(d_j-1)!}$.

This model is also of interest because of the duality relation,
\begin{align*}
    \mathfrak{Z}_q\bi{k_1,\dots,k_r}{d_1,\dots,d_r} = \mathfrak{Z}_q\bi{d_r,\dots,d_1}{k_r,\dots,k_1},
\end{align*}
which is exactly the partition relation for bi-brackets (\cite[Prop. 4]{Zu2}).

We consider this model not in more detail since it is \enquote{more or less} Bachmann's model of bi-brackets, and the connection to bi-brackets is also studied very well in \cite{Zu2}. 
\end{remark}

\begin{remark}[Comparison]
Bi-brackets are very good for connecting $q$MZVs direct with quasi-modular forms. Furthermore, this model gives with the partition relation another view onto SZ-duality since both are equivalent (cf. \cite[Thm. 3.22]{Bri}). However, explicit calculations are usually complicated in this model because the explicit representation of Eulerian polynomials is not very nice.
\end{remark}

\subsection{Takeyama-Bradley--Zhao model}

We saw in Proposition \ref{bz1X2} that the $\Q$-span of BZ-$q$MZVs is a proper subspace of $\mathcal{Z}_q$. But it is for some situations comfortable to extend the BZ-model of $q$MZVs, especially when the elements of the model should span $\mathcal{Z}_q$. 

This extension for the BZ-model is done as follows, as Takeyama did in \cite{Tak}:

\begin{definition}
Set $\eN := \left\{\oo\right\}\cup \N = \left\{\oo, 1, 2, 3,\dots\right\}$ and define for all $r\geq 1$, $k_1,\dots,k_r\in\eN$, $k_1\neq 1$,
\begin{align*}
    \tbz(k_1,\dots,k_r) := \sum\limits_{m_1>\dots>m_r>0} f(k_1,m_1)\dots f(k_r,m_r),
\end{align*}
where $f\left(\oo,m\right):= \frac{q^m}{1-q^m},\ f(k,m) := \frac{q^{(k-1)m}}{(1-q^m)^k}\ \text{for}\ k\geq 1.$ Define $\tbz(\emptyset) := 1$.
\end{definition}
As mentioned, this extension of the BZ-model spans $\mathcal{Z}_q$:
\begin{proposition}
\label{tbz span}
The TBZ-model spans $\mathcal{Z}_q$, i.e.,
\begin{align*}
    \mathcal{Z}_{q}=\left\langle \tbz(k_1,\dots,k_r)\, \middle|\, r\geq 0,\, k_i\in\eN,\, k_1\neq 1\right\rangle_\Q.
\end{align*}
In particular, the TBZ-model is closed under $q\frac{d}{d q}$.
\end{proposition}

\begin{proof}
The TBZ-model spans the same space as the SZ-model (Prop. \ref{tbz1}, Prop. \ref{extBZ-to-SZ2}), i.e., $\mathcal{Z}_q$.
\end{proof}

This extended version of the BZ-model satisfies a quasi-shuffle product that is compatible with the one of the non-extended model:

\begin{definition}
\begin{enumerate*}
    \item Define $\mathfrak{h}^{TBZ} := \Q\left\langle z_{\oo},z_1,z_2,\dots\right\rangle$. Then we can view $\tbz$ also as a map
    \begin{align*}
        \tbz: \mathfrak{h}^{TBZ} \longrightarrow \mathcal{Z}_q,\qquad
        z_{k_1}\dots z_{k_r}&\longmapsto\tbz(k_1,\dots,k_r),
    \end{align*}
    extended $\Q$-linearly and sending $\mathbf{1}\mapsto 1$.

\item Define on the alphabet $\Q\left\{ z_k\, |\, k\in\eN\right\}$ the associative and commutative product $\diamond_{TBZ}$ via
\begin{align*}
    z_{k_1}\diamond_{TBZ} z_{k_2} := z_{k_1 + k_2} + z_{k_1 + k_2 - 1},\quad z_k\diamond_{TBZ} z_{\oo} := z_{\oo}\diamond_{TBZ} z_k := z_{k+1}, \quad z_{\oo}\diamond_{TBZ} z_{\oo} := z_2 - z_1
\end{align*}
for all $k,k_1,k_2\in\N$. Furthermore, let $\ast_{TBZ}$ be the induced quasi-shuffle product on $\Q\left\langle z_{\oo},z_1,z_2,\dots\right\rangle$.
\end{enumerate*}
\end{definition}

Some straightforward computation shows that $\diamond_{TBZ}$ is indeed commutative and associative.

\begin{proposition}
\label{tbz1}
The map $\tbz$ is an algebra homomorphism, i.e., for all $u,v\in \Q\left\langle z_{\oo},z_1,z_2,\dots\right\rangle$ we have
\begin{align*}
    \tbz(u\ast_{TBZ} v) = \tbz(u) \tbz(v).
\end{align*}
\end{proposition}
\begin{proof}  
The proof is analogous to the one of Proposition \ref{bz qshuffle2}.
\end{proof} 

\begin{remark}
For the TBZ-model, no \enquote{good} generating series is known. This means that we don't know a one that could be written nicely or that would give us new results about TBZ-$q$MZVs.
\end{remark}

\begin{proposition}[Duality, {\cite[Thm. 4]{Tak}}]
\label{tbz dual2}
Let the maps $\Ubz$ and $\Vbz$ be defined as in Prop. \ref{extBZ-to-SZ2} resp. \ref{tbz transl2} below. Then, we have for all $w\in\mathfrak{h}^{TBZ}$
\begin{align*}
    \tbz(w) = \left(\tbz\circ \Vbz\circ \Tilde{\tau}\circ \Ubz\right) (w).
\end{align*}
\end{proposition}
\begin{proof}
Since $\Ubz$ and $\Vbz$ are translation maps of TBZ-$q$MZVs into the SZ-model resp. vice versa, the proof follows by SZ-duality.
\end{proof}

We give now an explicit translation into the SZ-model and vice versa:

\begin{proposition}[Translation TBZ into SZ, {\cite[Prop. 2.23]{Bri}}]
\label{extBZ-to-SZ2}
\ \\
For every $d_1,\dots, d_r\in\N_0$, $k_1,\dots,k_{r-1}\in\N$ with $k_1\geq 2$ if $d_1 = 0$ we have
\begin{align*}
    \tbz\left(\left\{\oo\right\}^{d_1},k_1,\dots,k_{r-1},\left\{\oo\right\}^{d_r}\right) 
    =  \sum\limits_{\substack{\delta_j\in\{0,1\}\\ 1\leq j\leq r-1}} \sz\left(\{1\}^{d_1}, k_1-\delta_1,\dots ,\{1\}^{d_{r-1}},k_{r-1}-\delta_{r-1}, \{1\}^{d_r}\right).
\end{align*}
Denote the corresponding map $\mathfrak{h}^{TBZ}\rightarrow \mathfrak{K}^1$ by $\Ubz$.
\end{proposition}

\begin{proof}
This is a direct consequence of the definition of SZ-$q$MZVs and the identity
\begin{align*}
    \frac{X^{k-1}}{(1-X)^k} = \frac{X^{k-1}}{(1-X)^{k-1}} + \frac{X^{k}}{(1-X)^{k}}\quad \text{for } k\in\N.\tag*{\qedhere}
\end{align*}
\end{proof}
For illustration, we give an example of this formula.
\begin{example}
\label{ex: tbz in sz2}
Consider $\tbz\left(\oo,2,1\right)$. We have
\begin{align*}
    \tbz\left(\oo,2,1\right) =\, & \sum\limits_{m_1>m_2>m_3>0} \frac{q^{m_1}}{1-q^{m_1}}\frac{q^{m_2}}{(1-q^{m_2})^2}\frac{1}{1-q^{m_3}}
    \\
    =\, & \sum\limits_{m_1>m_2>m_3>0} \frac{q^{m_1}}{1-q^{m_1}}\left(\frac{q^{m_2}}{1-q^{m_2}} + \frac{q^{2 m_2}}{(1-q^{m_2})^2}\right)\left(1+\frac{q^{m_3}}{1-q^{m_3}}\right)
    \\
    =\, & \sz(1,2,1) + \sz(1,1,1) + \sz(1,2,0) + \sz(1,1,0).
\end{align*}
\end{example}

\begin{proposition}[Translation SZ into TBZ, {\cite[Prop. 2.25]{Bri}}]
\label{tbz transl2}\ \\
For every $k_1,\dots,k_r\in\N,\ d_1,\dots,d_r\in\N_0$ we have
\begin{align*}
    &\, \sz\left(k_1,\{0\}^{d_1},\dots, k_r,\{0\}^{d_r}\right) 
    \\
    =\, & \sum\limits_{\substack{1\leq j_i\leq k_i,\\  \varepsilon_i\in\{\oo,1\}^{d_i}\\ 1\leq i\leq r}} (-1)^{\sum\limits_{i=1}^r k_i - j_i + |\varepsilon_i|} 
    \tbz\left(j_1\delta_{j_1\neq 1} + \oo \delta_{j_1=1},\varepsilon_1,\dots , j_r\delta_{j_r\neq 1} + \oo \delta_{j_r=1},\varepsilon_r\right),
\end{align*}
where $|\varepsilon|$ counts the $\oo$'s in $\varepsilon$; we denote the corresponding map $\mathfrak{K}^1\rightarrow\mathfrak{h}^{TBZ}$ by $\Vbz$.
\end{proposition}

\begin{proof}
This is a direct consequence of the identity
\begin{align*}
    1 = \frac{1-X}{1-X} = \frac{1}{1-X} - \frac{X}{1-X}
\end{align*}
for the zero-entries and
\begin{align*}
    \frac{X^k}{(1-X)^k} = \sum\limits_{1\leq j\leq k} (-1)^{k-j}\left[\frac{X^{j-1}}{(1-X)^j}\delta_{j\neq 1} + \frac{X}{1-X}\delta_{j=1}\right]\quad \text{for } k\in\N
\end{align*}
for the non-zero entries and using the definition of SZ- and TBZ-$q$MZVs. See \cite{Bri} for details.
\end{proof}

\begin{example}
\label{ex: sz in tbz2}
We have
\begin{align*}
    \sz(3,0,1) =\,  &\tbz\left(\oo,1,\oo\right)-\tbz\left(\oo,\oo,\oo\right)-\tbz\left(2,1,\oo\right)\\ &+\tbz\left(2,\oo,\oo\right)+\tbz\left(3,1,\oo\right)-\tbz\left(3,\oo,\oo\right).
\end{align*}
\end{example}

\begin{remark}[Comparison]
\begin{enumerate*}
    \item Comfortable about the TBZ-model is that its span is $\mathcal{Z}_q$ as for the most common models of $q$MZVs, which is good when we want to compare results in different models such as (SZ-)duality.
    \item It can be not comforting to work with $\mathfrak{h}^{TBZ}$ because of the extra letter $\oo$ since it doesn't guarantee that we can switch smoothly to $\mathfrak{h}^0$ or $\mathfrak{K}^1$, which are the underlying word algebras for all the other models we consider.
\end{enumerate*}
\end{remark}

\subsection{Ohno--Okuda--Zudilin model}
\label{ssec: OOZ2}

Another model for $q$-analogues of MZVs is the one first considered in 2012 (\cite{OOZ}) and named after Ohno, Okuda, and Zudilin. One application of this model is that some particular sum of OOZ-$q$MZVs is the generating series of the number of conjugacy classes of $\text{GL}(n,K)$ for a finite field $K$ (cf. \cite[§4.6]{Bri}).

\begin{definition}
\label{ooz defi2}
\begin{enumerate*} 
\item We will work with an extended version of OOZ-$q$MZVs: Define $\ooz(\emptyset) := 1$ and for $r\geq 1$ and all integers $k_1,\dots,k_r\geq 0$ with $k_1\geq 1$
\begin{align*}
    \ooz(k_1,\dots , k_r) :=\, \zeta_q(k_1,\dots,k_r;X,1\dots,1)
    =\, \sum\limits_{m_1>\dots>m_k>0} \frac{q^{m_1}}{(1-q^{m_1})^{k_1}\cdots (1-q^{m_r})^{k_r}}.
\end{align*}
When identifying an SZ-admissible index $(k_1,\dots,k_r)$ with $p^{k_1}y\cdots p^{k_r} y\in\mathfrak{K}^1$, we can define $\ooz$ also as the map, uniquely determined through $\mathbf{1}\mapsto 1$, $\Q$-linearity and 
\begin{align*}
    \ooz : \mathfrak{K}^1 \longrightarrow\mathcal{Z}_q,
    \qquad
    p^{k_1}y\dots p^{k_r}y \longmapsto \ooz(k_1,\dots,k_r).
\end{align*}

\item For a connection to the BZ-model, it is often useful to restrict to admissible indices: Hence, define
\begin{align*}
     \ooz : \mathfrak{h}^0 \longrightarrow\mathcal{Z}_q,
    \qquad
    z_{k_1}\dots z_{k_r} \longmapsto \ooz(k_1,\dots,k_r),
\end{align*}
extended to $\mathfrak{h}^0$ by $\Q$-linearity and map $\bone\mapsto 1$.

\end{enumerate*}
\end{definition}

Considering the span of both OOZ-models, the link to SZ- resp. BZ-$q$MZVs becomes clear:

\begin{proposition}
\label{ooz span2} For the $\Q$-span of the OOZ-model we have
\begin{enumerate*}
    \item  $ \mathcal{Z}_{q,1}=\left\langle \ooz(k_1,\dots,k_r)\, |\, r\geq 0,\, k_1\geq 2,\, k_i\geq 1\right\rangle_\Q$,
    \item $\mathcal{Z}_{q}=\left\langle \ooz(k_1,\dots,k_r)\, |\, r\geq 0,\, k_1\geq 1,\, k_i\geq 0\right\rangle_\Q$.
\end{enumerate*}
\end{proposition}

\begin{proof}  A proof can be obtained from Proposition \ref{ooz transl2}, where explicit translations of OOZ-$q$MZVs into SZ-$q$MZVs and vice versa (resp. restricted OOZ-$q$MZVs into BZ-$q$MZVs) are given.
\end{proof}

In particular, the OOZ-model we work with is closed under $q\frac{d}{d q}$ (Rem. \ref{qddq rem}(i)), while the restricted model is only conjecturally closed (Rem. \ref{qddq rem}(ii)). Also, the restricted OOZ-model satisfies a shuffle product, introduced in \cite{CEM}:

\begin{definition} 
Define the map $T:\mathfrak{h}^0 \rightarrow\mathfrak{h}^1$ via $\bone\mapsto 1$, $\Q$-linearity and for all $n\geq 2,\, v\in\mathfrak{h}^1$,
\begin{align*}
    z_n v\longmapsto z_n v - z_{n-1} v.
\end{align*}
Then the quasi-shuffle product $\sooz$ on $\mathfrak{h}^0$ for the OOZ-model is defined as the unique map $\mathfrak{h}^0\otimes_\Q \mathfrak{h}^0\rightarrow\mathfrak{h}^0$ satisfying for all $u,v\in\mathfrak{h}^0$
\begin{align*}
    T(u\sooz v) = T(u)\ast T(v).
\end{align*}
\end{definition}
The quasi-shuffle product $\shuffle_{OOZ}$ turns $\ooz$ into an homomorphism on $\mathfrak{h}^0$:
\begin{theorem}[{\cite[Section 4.5]{CEM}}]
The product $\sooz$ is well-defined and the evaluation map
\begin{align*}
    \ooz : (\mathfrak{h}^0,\sooz)\longrightarrow (\mathcal{Z}_q,\cdot),
    \qquad
    w\longmapsto \ooz(w)
\end{align*}
is an algebra homomorphism, i.e., in particular, we have for all $u,v\in\mathfrak{h}^0$
\begin{align*}
    \ooz(u\sooz v) = \ooz(u)\ooz(v).
\end{align*}
\end{theorem}

For details on the quasi-shuffle structure, OOZ-$q$MZVs imply, we refer to \cite{CEM} and \cite{EMS}.

As for other models, we consider a generating series for (the extended version of) OOZ-$q$MZVs:
\begin{proposition}[{\cite[Prop. A.89]{Bri}}]
Define for all $r\in\N_0$
\begin{align*}
    \mathfrak{oz}(X_1,\dots,X_r,Y_1,\dots,Y_r) := \sum\limits_{\substack{k_1,\dots,k_r\geq 1\\ d_1,\dots,d_r\geq 0}} \ooz\left(k_1,\{0\}^{d_1},\dots,k_r,\{0\}^{d_r}\right) \frac{X_1^{k_1}}{k_1!}Y_1^{d_1}\cdots\frac{X_r^{k_r}}{k_r!}Y_r^{d_r}.
\end{align*}
Then, we have with $m_{r+1} := 0$
\begin{align*}
    \mathfrak{oz}(X_1,\dots,X_r,Y_1,\dots,Y_r) = \sum\limits_{m_1>\dots>m_r>0} q^{m_1}\prod\limits_{j=1}^r (1+Y_j)^{m_j-m_{j+1}-1}\left(e^{\frac{X_j}{1-q^{m_j}}} - 1\right).
\end{align*}
\end{proposition}

\begin{proof}
We use the geometric sum identity and the binomial theorem to obtain
\begin{align*}
    &\, \mathfrak{oz}(X_1,\dots,X_r,Y_1,\dots,Y_r) 
    =
    \sum\limits_{\substack{k_1,\dots,k_r\geq 1\\ d_1,\dots,d_r\geq 0}} \sum\limits_{\substack{m_1>n_1>\dots>n_{d_1}\\ > m_2>\dots>n_{d_1+\dots+d_r}>0}} q^{m_1} \prod\limits_{j=1}^r \frac{1}{(1-q^{m_j})^{k_j}}\frac{X_j^{k_j}}{k_j!} Y_j^{d_j}
    \\
    =&\,
    \sum\limits_{\substack{m_1>\dots>m_r>0}} q^{m_1} \prod\limits_{j=1}^r \left(\sum\limits_{k_j\geq 1,\, d_j\geq 0} \frac{\left(\frac{X_j}{1+q^{m_j}}\right)^{k_j}}{k_j!} \binom{m_j-m_{j+1}-1}{d_j} Y_j^{d_j}\right)
    \\
    =&\,
    \sum\limits_{m_1>\dots>m_r>0} q^{m_1}\prod\limits_{j=1}^r (1+Y_j)^{m_j-m_{j+1}-1}\left(e^{\frac{X_j}{1-q^{m_j}}} - 1\right).\tag*{\qedhere}
\end{align*}
\end{proof}

For the OOZ-model, no individual duality relation is known so far. But we can translate into the SZ-model (resp. BZ-model for restricted definition), apply there SZ-duality (resp. BZ-duality) and then translate back into the OOZ-model, giving $\Q$-linear relations among OOZ-$q$MZVs:

\begin{theorem}[Duality, {\cite[Thm. 5.5, Thm. 5.9]{EMS}}]
\label{dual in ooz2}
Let be $\U,\, \V$ as in Proposition \ref{ooz transl2}.
\begin{enumerate*}
    \item For all $w\in\mathfrak{h}^0$ we have
$    \ooz(w) = \left(\ooz\circ \U^{-1}\circ\, \tau\circ \U\right) (w).
$
    \item For all $w\in\mathfrak{K}^1$ we have
$    \ooz(w) = \left(\ooz\circ \V^{-1}\circ\, \Tilde{\tau}\circ \V\right) (w).
$
    \item For any $w\in\mathfrak{K}^1$ we have 
    $\ooz(w) = \left(\szs\circ\Tilde{\tau}\right) (w),
    $
    where $\szs: \mathfrak{K}^1\rightarrow\, \mathcal{Z}_q$ is the map of multiple zeta star values, defined via $\bone\mapsto 1$, $\Q$-linearity and
    \begin{align*}
        p^{k_1}y\cdots p^{k_r}y\longmapsto\, \sum\limits_{m_1\geq \dots\geq m_r>0} \frac{q^{m_1 k_1}}{(1-q^{m_1})^{k_1}}\cdots \frac{q^{m_r k_r}}{(1-q^{m_r})^{k_r}}.
    \end{align*}
\end{enumerate*}
\end{theorem}

\begin{example}
Consider $w = p^2 y p y$, i.e., $\Tilde{\tau}(w) = p y p y y$. We get with SZ-duality
\begin{align*}
    &\ooz(w) =\, \ooz(2,1)
    \\ 
    =&\sum\limits_{m_1>m_2>0} \frac{q^{m_1}}{(1-q^{m_1})^2 (1-q^{m_2})}
    =\, 
    \sum\limits_{m_1>m_2>0} \left(\frac{q^{m_1}}{1-q^{m_1}} + \frac{q^{2 m_1}}{(1-q^{m_1})^2}\right) \left(1 + \frac{q^{ m_2}}{1-q^{m_2}}\right)
    \\
    =&\,
    \sz(1,0) + \sz(1,1) + \sz(2,0) + \sz(2,1)
    =\,
    \sz(2) + \sz(1,1) + \sz(2,0) + \sz(1,1,0)
    \\
    =&\, 
    \left(\sum\limits_{m_1=m_2=m_3>0} + \sum\limits_{m_1>m_2=m_3>0} + \sum\limits_{m_1=m_2>m_3>0} + \sum\limits_{m_1>m_2>m_3>0} \right) \frac{q^{m_1}}{1-q^{m_1}} \frac{q^{m_2}}{1-q^{m_2}}
    \\
    =&\,
    \sum\limits_{m_1\geq m_2\geq m_3 > 0} \frac{q^{m_1}}{1-q^{m_1}} \frac{q^{m_2}}{1-q^{m_2}}
    =
    \szs(1,1,0)
    =
    \szs(\Tilde{\tau}(w)).
\end{align*}
\end{example}

As mentioned, we give now the maps $\U$ resp. $\V$, proving in particular Proposition \ref{ooz span2}.

\begin{proposition}[{\cite[Prop. 5.7, Rem. 5.8]{EMS}}]
\label{ooz transl2}
\begin{enumerate*}
\item The map $\U:\mathfrak{h}^0\rightarrow \mathfrak{h}^0$, defined through $\bone\mapsto 1$,
\begin{align*}
    z_{k_1}\dots z_{k_r}\longmapsto \sum\limits_{\substack{2\leq n_1\leq k_1\\ 1\leq n_j\leq k_j,\, j\geq 2}} \binom{k_1 -2}{n_1 - 2}\binom{k_2 - 1}{n_2 - 1}\cdots \binom{k_r - 1}{n_r - 1} z_{n_1}\dots z_{n_r}
\end{align*}
and $\Q$-linear continuation, is a linear isomorphism and satisfies $\ooz = \bz\circ \U$.
\item Analogously, the map $\V:\mathfrak{K}^1\rightarrow \mathfrak{K}^1$, given by $\bone\mapsto 1$,
\begin{align*}
    p^{k_1}y\dots p^{k_r}y\longmapsto \sum\limits_{\substack{1\leq n_1\leq k_1\\ 0\leq n_j\leq k_j,\, j\geq 2}} \binom{k_1 -1}{n_1 - 1}\binom{k_2}{n_2}\cdots \binom{k_r}{n_r} p^{n_1}y\dots p^{n_r}y
\end{align*}
and $\Q$-linear continuation is a linear isomorphism and 
we have $\ooz = \sz\circ \V$.
\end{enumerate*}
\end{proposition}

\begin{remark}[Comparison]
\begin{enumerate*}
    \item The \enquote{duality relations} (i) and (ii) can be indeed viewed as duality in the OOZ-model, the authors in \cite{EMS} did so. Still, we should compare them with the partition relation in the model of bi-brackets: The partition relation is the same as when translating bi-brackets into SZ-$q$MZVs, applying SZ-duality and then translating back into bi-brackets.
    \item Duality relation (iii) is no \enquote{real} duality relation, but in \cite{EMS} called so, since $\Tilde{\tau}$ gives SZ-duality in the SZ-model. However, this relation is interesting because it is, in truth, the translation map of the OOZ-model into the SZS-model, and this translation is obtained by the \enquote{duality map} $\Tilde{\tau}$.
\end{enumerate*}
\end{remark}

\subsection{Okounkov $q$MZVs}

In the context of enumerative geometry and Hilbert schemes, a model of $q$MZVs introduced by Okounkov (see \cite{Oko}) is very important. 

\begin{definition}[Okounkov $q$MZVs]
\label{oko defi}
Define for all indices $\mathbf{k} = (k_1,\dots,k_r)$ ($r\in\N_0$) with every entry at least $2$ the Okounkov-$q$MZV of $\mathbf{k}$,
\begin{align*}
    \oko(\mathbf{k}) :=\, \oko(k_1,\dots,k_r) :=\, \zeta_q\left(k_1,\dots,k_r;p_{k_1},\dots,p_{k_r}\right)
    =\,  \sum\limits_{m_1>\dots>m_r>0} \prod\limits_{j=1}^r \frac{p_{k_j}(q^{m_j})}{(1-q^{m_j})^{k_j}},
\end{align*}
where $\oko(\emptyset):= 1$ as usually and 
 $   p_k(X) := \begin{cases}
    X^{\frac{k}{2}},\ &\text{if } k\ \text{is even},\\
    X^{\frac{k-1}{2}}(1+X),\ &\text{if } k\ \text{is odd}.
    \end{cases}
$
\end{definition}
The span of the Okounkov-model is only a subspace of $\mathcal{Z}_q$:
\begin{proposition}[{\cite[§2 (iv)]{BK2}}]
\label{oko span2}
We have
 $   \mathcal{Z}_{q,1}^\circ=\, \left\langle \oko(k_1,\dots,k_r)\, |\, r\geq 0,\, k_i\geq 2\right\rangle_\Q.$
In particular, the span of Okounkov-$q$MZVs is conjecturally closed under $q\frac{d}{dq}$ (Rem. \ref{qddq rem} (ii)).
\end{proposition}

As the other models of $q$MZVs, also Okounkov-$q$MZVs satisfy a quasi-shuffle product:

\begin{definition}
\begin{enumerate*}
\item Consider $\mathfrak{h}^{Oko} := \Q\left\langle z_2,z_3\dots\right\rangle$ and the map $\oko : \mathfrak{h}^{Oko} \rightarrow \mathcal{Z}_q$, defined via $\Q$-linear continuation, $\mathbf{1}\mapsto 1$ and
\begin{align*}
    z_{k_1}\dots z_{k_r} \longmapsto \oko(k_1,\dots,k_r).
\end{align*}

\item Define on the alphabet $A_{Oko} := \Q\{z_2,z_3,\dots\}$ the following commutative and associative product by $\Q$-bilinearity, $\bone\diamond_{Oko} w := w\diamond \bone$ for all $w\in A_{Oko}$ and
\begin{align*}
    (z_{k_1}, z_{k_2}) \longmapsto \begin{cases}
    z_{k_1+k_2-1} + z_{k_1+k_2+1},\quad &\text{if } k_1,\, k_2\ \text{is odd},\\
    z_{k_1+k_2},\quad &\text{else}.
    \end{cases}
\end{align*}
Furthermore, let be $\ast_{Oko}$ the induced quasi-shuffle product on $\mathfrak{h}^{Oko}$.
\end{enumerate*}
\end{definition}

\begin{proposition}
The map $\oko$ is an algebra homomorphism, i.e., we have for all $w_1,w_2\in\mathfrak{h}^{Oko}$
\begin{align*}
    \oko(w_1\ast_{Oko} w_2) = \oko(w_1)\oko(w_2).
\end{align*}
\end{proposition}

\begin{proof}
The claim follows by the definition of Okounkov-$q$MZVs as iterated sums and if at least one of $k_1$ and $k_2$ is even, then we have $p_{k_1}(X)p_{k_2}(X) = p_{k_1+k_2}(X)$, such as if both are odd, we have
\begin{align*} 
p_{k_1}(X)p_{k_2} (X) = \left(X^{\frac{(k_1 + k_2-1)-1}{2}} + X^{\frac{(k_1 + k_2+1)-1}{2}}\right)(1+X)  = p_{k_1+k_2-1}(X)+p_{k_1+k_2+1}(X).\tag*{\qedhere} 
\end{align*}

\end{proof}

\begin{remark}
For the Okounkov model, there is no \enquote{good} generating series 
known, i.e., no one that has a nice representation or would lead to further results. However, we can introduce for all $r\geq 1$ one:
\begin{align*}
    \gok(X_1,\dots,X_r) :=&\, \sum\limits_{k_1,\dots,k_r\geq 2} \oko(k_1,\dots,k_r) X_1^{k_1 - 2}\cdots X_r^{k_r - 2}
    \\
    =&\, \sum\limits_{m_1>\dots>m_r>0} \prod\limits_{j=1}^r \frac{q^{m_j}(1-q^{m_j}+(1+q^{m_j})X_j)}{(1-q^{m_j})((1-q^{m_j})^2-q^{m_j}X_j)}.
\end{align*}
\end{remark}

\begin{remark}
Since the Okounkov model does not span the same space as the SZ- or BZ-model, we cannot 
translate into the respective model, apply there SZ- or BZ-duality, and translate back as we did in the OOZ-model. Also, no individual duality relation for the Okounkov model is known so far.
\end{remark}

\begin{remark}[Comparison]
For considering duality relations or related topics, the Okounkov model is not the best choice. But the deep connection to Hilbert schemes of points (\cite{Oko}) makes this model essential and is for this well suited. Also, direct calculations are pretty good to do in this model because of the \enquote{almost} monomial polynomials $p_k$.
\end{remark}

\section{Subalgebras of $\mathcal{Z}_q$}

There are several subalgebras of $\mathcal{Z}_q$ that are of interest. One of the most important is the algebra of quasi-modular forms. Others get their importance by conjectures that they are not only subalgebras but also equal $\mathcal{Z}_q$, which would give - assuming that they are true - a much deeper understanding of the structure of $\mathcal{Z}_q$. They are all verified for small weights, often by computer assistance. Bachmann for example considered, before he introduced bi-brackets, brackets and their algebra
\begin{align*}
    \mathcal{MD} := \left\langle \g(k_1,\dots,k_r)\, |\, r\geq 0,\, k_1\geq 2,\, k_i\geq 1\right\rangle_\Q.
\end{align*}

The algebra $\mathcal{MD}$ contains the classical Eisenstein series $G_2,\, G_4,\, G_6$ because of 
\begin{align*}
    G_2 = -\frac{1}{24} + \g(2),\quad G_4 = \frac{1}{1440}  + \g(4),\quad G_6 = -\frac{1}{60480} + \g(6),
\end{align*}
which is (with \cite{KZ}) why the ring of quasi-modular forms $\widetilde{M}(\SL_2(\Z))_\Q$ is contained in $\mathcal{MD}$,
\begin{align*}
    \widetilde{M}(\SL_2(\Z))_\Q =\Q[G_2,G_4,G_6]\subset \mathcal{MD}
\end{align*}
(notice that this is a proper inclusion since, e.g., $\g(2,1)\neq 0$ has odd weight; hence, $\g(2,1)$ is not algebraic over $\Q$ in terms of $G_2,\, G_4,\, G_6$). We mentioned $\mathcal{MD}$ already but with different name:

\begin{proposition}[{\cite[Thm. 2.3 (ii)]{BK2}}]
We have $\mathcal{MD} = \mathcal{Z}_q^\circ$. Particularly, $\mathcal{MD}$ is stable under $q\frac{d}{d q}$.
\end{proposition}

For the $\Q$-algebra of bi-brackets (often denoted by $\mathcal{BD}$) is proven that $\mathcal{BD} = \mathcal{Z}_q$). By comparing dimensions in small weight, Bachmann conjectured that brackets and bi-brackets span the same space:

\begin{conjecture}[{\cite[Conj. 4.3]{Ba4}}]
We have
$\mathcal{MD} = \mathcal{BD}.$, i.e., $\mathcal{Z}_q^\circ = \mathcal{Z}_q$.
\end{conjecture}

There are more subalgebras of $\mathcal{Z}_q$ that are often considered:

\begin{proposition}
We have
$ 
\mathcal{Z}_{q,1}^\circ = \left\langle \g(\mathbf{k})\, |\, k_i\geq 2 \right\rangle_\Q = \left\langle \oko(\mathbf{k})\, | \, k_i\geq 2\right\rangle_\Q = \left\langle \bz(\mathbf{k})\, |\, k_i\geq 2 \right\rangle_\Q.
$
\end{proposition}

\begin{proof}
This is \cite[Thm. 2.3 (iii)]{BK2}, Proposition \ref{oko span2} and an analogous proof of Proposition \ref{bz1X2}.
\end{proof}

We get other important subalgebras of $\mathcal{Z}_q$ when considering bi-brackets. By defining the weight and depth as done, we get a filtration by weight and depth on $\mathcal{Z}_q$ (resp. on every subalgebra of $\mathcal{Z}_q$):

\begin{definition}[\cite{Ba4}]
Let be $A$ a subalgebra of $\mathcal{Z}_q$ and $r,s\geq 0$. Define
\begin{enumerate*}
    \item the \emph{weight filtration} $\Fil_r^W(A) := \left\langle b=\g\bi{k_1,\dots,k_s}{d_1,\dots,d_s}\in A\Big|\, 0\leq s\leq r,\ \wt(b)\leq r\right\rangle_\Q$,
    \item the \emph{depth filtration} $\Fil_k^D(A) := \left\langle b=\g\bi{k_1,\dots,k_s}{d_1,\dots,d_s}\in A\Big|\, 0\leq s\leq r\right\rangle_\Q$,
    \item $\Fil_{r,s}^{W,D}(A) := \Fil_r^W \Fil_s^D (A)$
\end{enumerate*}
and denote by $\gr_r^W$, resp. $\gr_{r,s}^{W,D}$ the associated graded $\Q$-vector spaces.
\end{definition}

For the dimensions of the graded parts of $\mathcal{Z}_q$, Bachmann and K\"uhn give in \cite{BK2} conjectures standing in analogy to the one by Zagier and Broadhurst--Kreimer. Hence, for completeness, we want to state the latter ones first. We use for all $k\geq 2,\, n\geq 2,\, d\geq 0$ the notation
\begin{align*}
    \mathcal{Z} := \left\langle \zeta(\mathbf{k})|\, \mathbf{k}\ \text{admissible}\right\rangle_\Q,\ \ 
    \mathcal{Z}_k := \left\langle \zeta(\mathbf{k})|\, \wt(\mathbf{k})=k\right\rangle_\Q,\ \ 
    \mathcal{Z}_n^d := \left\langle \zeta(\mathbf{k})|\, \wt(\mathbf{k})=n,\, \depth(\mathbf{k})=d\right\rangle_\Q.
\end{align*}

\begin{conjecture}
\begin{enumerate*} 
\item (Zagier). We have $\mathcal{Z} = \bigoplus\limits_{k\geq 0} \mathcal{Z}_k$. Define $d_k$ via
\begin{align*}
    \sum\limits_{k\geq 0} d_k X^k = \frac{1}{1-X^2-X^3} = \frac{1}{1-x^2} \frac{1}{1-O_3(x)}
\end{align*}
with $O_3(X):=\frac{X^3}{1-X^2}$. Then we have $\dim_\Q(\mathcal{Z}_k) = d_k$. 
\item (Hoffman). The MZVs of indices containing only 2's and 3's build a basis of $\mathcal{Z}$.
\end{enumerate*}
\end{conjecture}
Remark that Hoffman's conjecture is stronger than Zagier's and is in accordance with Brown's theorem saying that MZVs with indices containing only 2's and 3's generate $\mathcal{Z}$, i.e., $\dim_\Q(\mathcal{Z}_k)\leq d_k$.

\begin{conjecture}[Broadhurst--Kreimer \cite{BK}]
With $E_2(X):=\frac{X^2}{1-X^2}$, $S(X):=\frac{X^{12}}{(1-X^4)(1-X^6)}$, we have
\begin{align*}
    1+ \sum\limits_{n\geq 1,\, d\geq 1} \dim_\Q\left(\faktor{\mathcal{Z}_n^d}{\mathcal{Z}_n^{d-1}}\right) X^n Y^d = \frac{1+E_2(X)Y}{1-O_3(X)Y+S(X)Y^2-S(X)Y^4}.
\end{align*}
\end{conjecture}

\begin{conjecture}[{\cite[Conj. 1.3]{BK2}}]
\begin{enumerate*}
    \item The dimensions of the weight graded parts of $\mathcal{Z}_q$ are given through
    \begin{align*}
        &\sum\limits_{k\geq 0} \dim_\Q\left(\gr_k^W \mathcal{Z}_q\right) X^k=\,  \frac{1}{1-X-X^2-X^3+X^6+X^7+X^8+X^9}
        \\
        =&\, \frac{1}{(1-X^2)(1-X^4)(1-X^6)} \frac{1}{1-D(X) O_1(X)+D(X)(E_4(X))},
    \end{align*}
    where we set $D(X):= \frac{1}{1-X^2}$, $O_1(X):=\frac{X}{1-X^2}$, $E_4(X) := \frac{X^4}{1-X^2}$.
    \item $\mathcal{Z}_q$ is generated by bi-brackets of indices only containing 1's, 2's and 3's.
    \item For the weight and depth graded parts of $\mathcal{Z}_q$, we have
    \begin{align*}
        \, \sum\limits_{k,\, \ell\geq 0} \dim_\Q\left(\gr_{k,\ell}^{W,D} \mathcal{Z}_q\right) X^k Y^\ell 
        =\, \frac{1+D(X)E_2(X)Y+D(X)S(X)Y^2}{1-a_1(X)Y+a_2(X)Y^2-a_3(X)Y^3-a_4(X)Y^4+a_5(X)Y^5}
    \end{align*}
    with
    \begin{align*}
        &a_1(X) := D(X)O_1(X),\ a_2(X):= D(X)\sum\limits_{k\geq 1} \dim_\Q(M_k(\SL_2(\Z)))^2 X^k,
        \\
        &a_4(X)=a_5(X):= O_1(X)S(X),\ a_4(X) := D(X)\sum\limits_{k\geq 1} \dim_\Q(S_k(\SL_2(\Z)))^2 X^k.
    \end{align*}
\end{enumerate*}
\end{conjecture}

Okounkov conjectured the following about the structure of $q$MZVs he introduced:

\begin{conjecture}[{\cite[Conjecture 1]{Oko}}]
\begin{enumerate*}
    \item We have
    \begin{align*}
    \sum\limits_{k\geq 0} \dim_\Q \left(\gr_k^w\left(\mathcal{Z}_q^{Oko}\right)\right) t^k =&\, \frac{1}{1 - t^2 - t^3 - t^4 - t^5 + t^8 + t^9 + t^{10} + t^{11} + t^{12}} 
    \\
    =&\, \frac{1}{(1-t^2)(1-t^4)(1-t^6)}\cdot \frac{1}{1-D(t)O_3(t)+2D(t)S(t)}.
\end{align*}
\item The space $\qmzvok$ is spanned by the $\oko(\mathbf{k})$ with $2\leq k_i\leq 5$.
\end{enumerate*}
\end{conjecture}

Finally, we now give an overview of the diverse, most commonly considered, subalgebras of $\mathcal{Z}_q$.

\begin{remark}[to the overview] 
\begin{enumerate*}
    \item We denote by $\mathbf{MF}$ the algebra of modular forms, where we take a modular form formally via its Fourier expansion in canonical way as element of the bigger algebras.
    \item The notation $q\frac{d}{d q}$ indicates that the respective algebra is closed under $q \frac{d}{d q}$. Dashed arcs with question marks mean that it is not proven yet but conjectured.
    \item The equality signs with a question mark mean that equality is conjectured but not proven yet.
    \item Equalities in blue boxes are just different notations from different papers for the same algebra.
    \item Indices $\mathbf{k}$ can have arbitrary non-negative length. For the sake of clarity, this is not extra mentioned.
\end{enumerate*}
\end{remark}

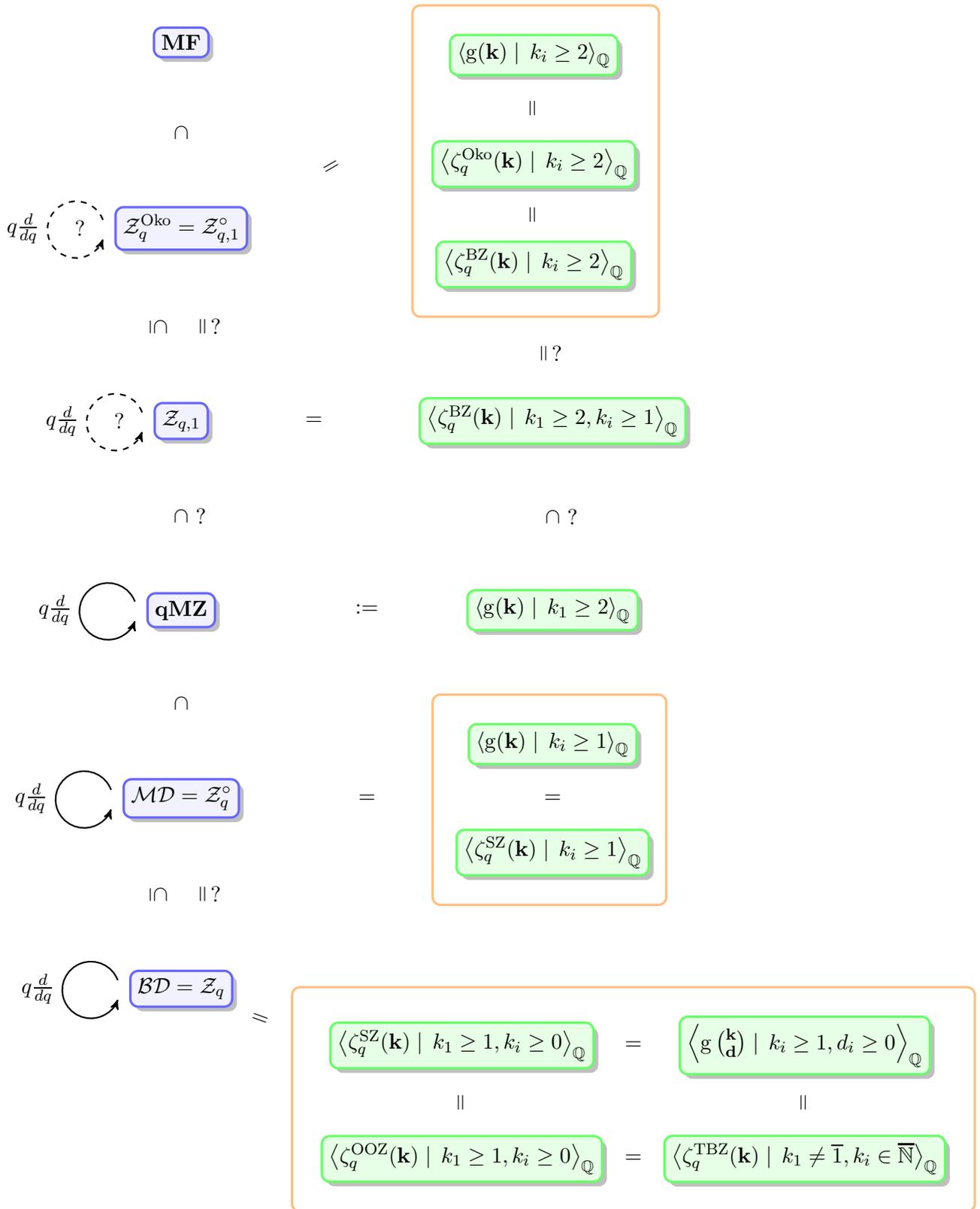
\begin{figure}
    \centering
    \begin{tikzpicture}[
ellip/.style={rectangle, rounded corners, draw=green!60, fill=green!10, very thick, minimum size=7mm, drop shadow},
ellipr/.style={rectangle, rounded corners, draw=red!60, fill=red!10, very thick, minimum size=7mm, drop shadow},
squarenode/.style={rectangle, draw=red!60, fill=red!5, very thick, minimum size=5mm, rounded corners, drop shadow},
squarenodeb/.style={rectangle, draw=blue!60, fill=blue!15, very thick, minimum size=5mm, rounded corners, drop shadow},
squarenodebl/.style={rectangle, rounded corners, draw=blue!60, fill=blue!5, very thick, minimum size=5mm, drop shadow},
squarenodelbl/.style={rectangle, rounded corners, draw=black!60, fill=mycyan, very thick, minimum size=5mm, drop shadow},
every text node part/.style={align=center}]

\node[squarenodebl] at (0,0) (MF) {$\mathbf{MF}$};
\node[squarenodebl] at ($(MF.south)+(0,-3)$) (qmzv) {$\qmzvok=\mathcal{Z}_{q,1}^\circ$};
\node[squarenodebl] at ($(qmzv.south)+(0,-3)$) (bz) {$\mathcal{Z}_{q,1}$};
\node[squarenodebl] at ($(bz.south)+(0,-3)$) (qmd) {$\text{\textbf{qMZ}}$};
\node[squarenodebl] at ($(qmd.south)+(0,-3)$) (MD) {$\mathcal{MD}=\mathcal{Z}_q^\circ$};
\node[squarenodebl] at ($(MD.south)+(0,-3)$) (Zq) {$\mathcal{BD}=\mathcal{Z}_q$};

\node at ($(qmzv.east)+(5,0)$) (bz1) {};

\node[ellip] at ($(bz1.south)+(0,-0.5)$) (bz2) {$\left\langle \bz(\mathbf{k})\mid\, k_i\geq 2\right\rangle_\Q$};
\node[ellip] at ($(bz1.north)+(0,1)$) (oko) {$\left\langle \oko(\mathbf{k})\mid\, k_i\geq 2\right\rangle_\Q$};
\node[ellip] at ($(oko.north)+(0,1.5)$) (g2) {$\left\langle \g(\mathbf{k})\mid\, k_i\geq 2\right\rangle_\Q$};
\node[rotate=90] at ($0.5*(bz2.north)+0.5*(oko.south)$) () {$=$};
\node[rotate=90] at ($0.5*(oko.north)+0.5*(g2.south)$) () {$=$};
\draw[draw=orange!50, rounded corners, very thick] ($(bz2.south west) + (-0.4,-0.5)$) rectangle ($(g2.north east) + (0.7,0.5)$);

\node[rotate=-90] at ($0.5*(MF.south)+0.5*(qmzv.north)$) () {$\subset$};

\node[rotate=-90] at ($0.5*(qmzv.south)+0.5*(bz.north)+(-0.4,0)$) () {$\subseteq$};
\node[rotate=90] at ($0.5*(qmzv.south)+0.5*(bz.north)+(0.4,0)$) (eq2) {$=$};
\node at (eq2.south) () {?};

\node[rotate=90] at ($0.5*(bz.south)+0.5*(qmd.north)$) (suBa4) {$\supset$};
\node at ($(suBa4.south)+(0.1,0)$) () {?};

\node[rotate=-90] at ($0.5*(qmd.south)+0.5*(MD.north)$) () {$\subset$};

\node[rotate=-90] at ($0.5*(MD.south)+0.5*(Zq.north)+(-0.4,0)$) () {$\subseteq$};
\node[rotate=90] at ($0.5*(MD.south)+0.5*(Zq.north)+(0.4,0)$) (eq3) {$=$};
\node at (eq3.south) () {?};

\node[rotate=30] at ($0.7*(qmzv.north east)+0.3*(bz1.west)+(0,0.8)$) () {$=$};

\draw [->,thick,dashed] ($(bz.west) + (-0.2,0.15)$) arc (20:340:0.5);
\node at ($(bz.west)+(-1.6,0)$) () {$q\frac{d}{dq}$};
\node at ($(bz.west)+(-0.6,0)$) () {?};

\draw [->,thick,dashed] ($(qmzv.west) + (-0.2,0.15)$) arc (20:340:0.5);
\node at ($(qmzv.west)+(-1.6,0)$) () {$q\frac{d}{dq}$};
\node at ($(qmzv.west)+(-0.6,0)$) () {?};

\draw [->,thick] ($(qmd.west) + (-0.2,0.15)$) arc (20:340:0.5);
\node at ($(qmd.west)+(-1.6,0)$) () {$q\frac{d}{dq}$};

\draw [->,thick] ($(MD.west) + (-0.2,0.15)$) arc (20:340:0.5);
\node at ($(MD.west)+(-1.6,0)$) () {$q\frac{d}{dq}$};

\draw [->,thick] ($(Zq.west) + (-0.2,0.15)$) arc (20:340:0.5);
\node at ($(Zq.west)+(-1.6,0)$) () {$q\frac{d}{dq}$};

\node[ellip] at ($(bz)+(6.5,0)$) (eg1) {$\left\langle \bz(\mathbf{k})\mid\, k_1\geq 2,k_i\geq 1\right\rangle_\Q$};

\node[ellip] at ($(qmd)+(6.5,0)$) (eg2) {$\left\langle \g(\mathbf{k})\mid\, k_1\geq 2\right\rangle_\Q$};
\node at ($0.5*(qmd.south)+0.5*(eg2.north)$) (g2) {$:=$};

\node at ($(MD)+(6.5,0)$) (eqMD) {$=$};

\node[ellip] at ($(eqMD.north)+(0,0.75)$) (eg3) {$\left\langle \g(\mathbf{k})\mid\, k_i\geq 1\right\rangle_\Q$};

\node[ellip] at ($(eqMD.south)+(0,-0.75)$) (sz1) {$\left\langle \sz(\mathbf{k})\mid\, k_i\geq 1\right\rangle_\Q$};

\node at ($0.5*(MD)+0.5*(eqMD)$) () {$=$};

\node at ($0.5*(bz.east)+0.5*(eg1.west)$) (g1) {$=$};
\node[rotate=90] at ($0.4*(bz2.south)+0.6*(eg1.north)$) (=?) {$=$};
\node at (=?.south) () {?};

\node[rotate=90] at ($0.5*(eg1.south)+0.5*(eg2.north)$) (suBa3) {$\supset$};
\node at ($(suBa3.south)+(0.1,0)$) () {?};

\draw[draw=orange!50, rounded corners, very thick] ($(sz1.south west) + (-0.4,-0.5)$) rectangle ($(eg3.north east) + (0.5,0.5)$);

\node at ($(Zq.east)+(4,-2)$) (phantom) {};

\node[ellip] at ($(phantom) + (0,1)$) (fg4) {$\left\langle \sz(\mathbf{k})\mid\, k_1\geq 1,k_i\geq 0\right\rangle_\Q$};
\node[rotate=90] at (phantom) (h4) {$=$};

\node[ellip] at ($(fg4)+(6,0)$) (eg4) {$\left\langle \g\bi{\mathbf{k}}{\mathbf{d}}\mid\, k_i\geq 1,d_i\geq 0\right\rangle_\Q$};

\node[ellip] at ($(phantom)+(0,-1)$) (fg3) {$\left\langle \ooz(\mathbf{k})\mid\, k_1\geq 1,k_i\geq 0\right\rangle_\Q$};
\node at ($0.5*(eg4)+0.5*(fg4)$) (h3) {$=$};

\node[ellip] at ($(fg3)+(6,0)$) (fg2) {$\left\langle \tbz(\mathbf{k})\mid\, k_1\neq \oo,k_i\in\eN\right\rangle_\Q$};
\node at ($0.5*(fg3)+0.5*(fg2)$) (h2) {$=$};

\node[rotate=90] at ($0.5*(eg4)+0.5*(fg2)$) () {$=$};

\gettikzxy{(fg2.south west)}{\bzx}{\bzy};
\gettikzxy{(eg4.north east)}{\gx}{\gy};
\gettikzxy{(fg3.west)}{\woozx}{\woozy};
\gettikzxy{(fg2.east)}{\eoozx}{\eoozy};
\gettikzxy{(Zq)}{\Zqx}{\Zqy};

\draw[draw=orange!50, rounded corners, very thick] ($(\woozx,\bzy) + (-0.5,-0.5)$) rectangle ($(\eoozx,\gy) + (0.5,0.5)$);

\node[rotate=-20] at ($0.7*(Zq.south east) + 0.3*(\woozx,\Zqy)+(0,-0.3)$) () {$=$};

\end{tikzpicture}
\end{figure}
\pagebreak
\section{$q$MZVs as generating functions of marked partitions}

We give in this section a combinatorial view of the considered dualities of $q$-analogues of MZVs. For that, we will use partitions intensively. A good reference on partitions, in general, is, for example, \cite{FH}. For Stanley coordinates, we refer to Stanley's original work \cite{Sta}.
\\ \ 

A partition of a natural number $N$ is usually defined as an in descending order sorted tuple of natural numbers $\lambda = (\lambda_1,\dots,\lambda_h)$ (i.e., $\lambda_1\geq \dots\geq\lambda_h$) with
\begin{align*}
    |\lambda| := \lambda_1 + \dots + \lambda_h = N.
\end{align*}
We often write $\lambda \dashv N$ to say that $\lambda$ is partition of $N$.

We can characterise $\lambda\dashv N$ also in a different way via summarizing the $\lambda_i$ that are equal: So we can identify $\lambda$ with two tuples of natural numbers, $\mathbf{m} = (m_1,\dots,m_r)$ and $\mathbf{n} = (n_1,\dots,n_r)$, where $\mathbf{m}$ contains the in $\lambda$ appearing values in strict descending order (i.e., $m_1 > \dots > m_r > 0)$ and $\mathbf{n}$ their multiplicities, i.e., $n_i$ describes the number of $\lambda_j$ being equal to $m_i$ and one has $N = \sum_{j=1}^r m_j n_j$. 

\begin{definition}[\emph{Stanley Coordinates}]
A partition $\mathbf{p}$ of length $r$ of some $N\in\N$ in Stanley coordinates is a pair of $r$ natural numbers $(\mathbf{m},\mathbf{n})=((m_1,\dots,m_r),(n_1,\dots,n_r))$ such that
\begin{enumerate*}
    \item $m_1>\dots>m_r>0$,
    \item $m_1 n_1 + \dots + m_r n_r = N$.
\end{enumerate*}
\end{definition}

By $\mathbf{p}'$ the \emph{transposed partition} of $\mathbf{p}$ is denoted, i.e., the one with Young diagram reflected at the main diagonal. Formally, if $\mathbf{p} = (\mathbf{m},\mathbf{n})\dashv N$  is a partition of a natural number $N$, we set
\begin{align*}
    \mathbf{p}' := ((n_1+\dots+n_r,\dots,n_1+n_2,n_1),(m_r,m_r-m_{r-1}\dots,m_1-m_2))\dashv N.
\end{align*}

The following illustrates the formal definition of the transposed partition:

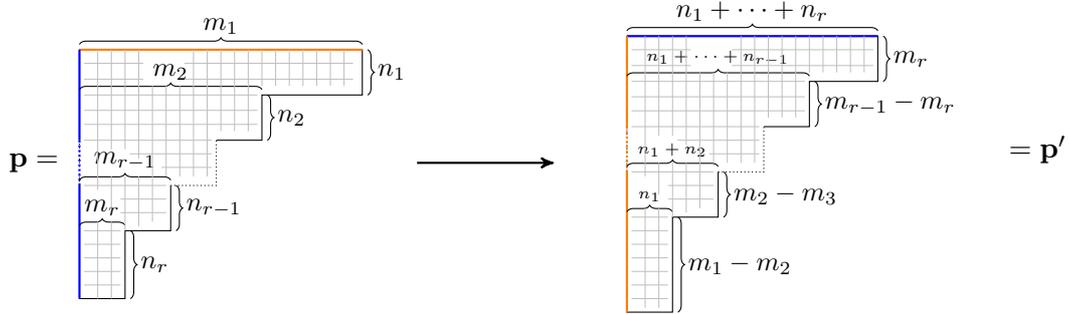
\begin{center}
\begin{tikzpicture}[scale=0.6]
\draw (-1,3) node{$\mathbf{p} = $};
\draw (0,2.5) -- (0,0) -- (1,0) -- (1,1.5) -- (2,1.5) -- (2,2.5);
\draw [densely dotted] (0,2.5) -- (0,3.5);
\draw [densely dotted] (2,2.5) -- (3,2.5) -- (3,3.5);
\draw (3,3.5) -- (4,3.5) -- (4,4.5) -- (6.2,4.5) -- (6.2,5.5) -- (0,5.5) -- (0,3.5);

\draw [thick, orange] (0,5.5) -- (6.2,5.5); 
\draw [thick, blue] (0,5.5) -- (0,3.5); 
\draw [thick, blue, densely dotted] (0,2.5) -- (0,3.5);
\draw [thick, blue] (0,0) -- (0,2.5);

\draw[decoration={brace,raise=2pt},decorate] (0,5.5) -- node[above=2pt] {\small $m_1$} (6.2,5.5);
\draw[decoration={brace,raise=2pt},decorate] (0,4.5) -- node[above=2pt] {\small $m_2$} (4,4.5);
\draw[decoration={brace,raise=2pt},decorate] (0,2.5) -- node[above=2pt] {\small $m_{r-1}$} (2,2.5);
\draw[decoration={brace,raise=2pt},decorate] (0,1.5) -- node[above=2pt] {\small $m_{r}$} (1,1.5);

\draw[decoration={brace,raise=2pt},decorate] (6.2,5.5) -- node[right=2pt] {\small $n_{1}$} (6.2,4.5);
\draw[decoration={brace,raise=2pt},decorate] (4,4.5) -- node[right=2pt] {\small $n_{2}$} (4,3.5);
\draw[decoration={brace,raise=2pt},decorate] (2,2.5) -- node[right=2pt] {\small $n_{r-1}$} (2,1.5);
\draw[decoration={brace,raise=2pt},decorate] (1,1.5) -- node[right=2pt] {\small $n_{r}$} (1,0);

\draw [lightgray,ultra thin] (4,5.5) -- (4,4.7);
\draw [lightgray,ultra thin] (4.3,5.5) -- (4.3,4.7);
\draw [lightgray,ultra thin] (4.6,5.5) -- (4.6,4.7);
\draw [lightgray,ultra thin] (4.9,5.5) -- (4.9,4.7);
\draw [lightgray,ultra thin] (5.2,5.5) -- (5.2,4.7);
\draw [lightgray,ultra thin] (5.5,5.5) -- (5.5,4.7);
\draw [lightgray,ultra thin] (5.8,5.5) -- (5.8,4.7);
\draw [lightgray,ultra thin] (3.7,5.5) -- (3.7,3.7);
\draw [lightgray,ultra thin] (3.4,5.5) -- (3.4,3.7);
\draw [lightgray,ultra thin] (3.1,5.5) -- (3.1,3.7);
\draw [lightgray,ultra thin] (2.8,5.5) -- (2.8,2.7);
\draw [lightgray,ultra thin] (2.5,5.5) -- (2.5,2.7);
\draw [lightgray,ultra thin] (2.2,4.7) -- (2.2,2.7);
\draw [lightgray,ultra thin] (1.9,4.7) -- (1.9,2.7);
\draw [lightgray,ultra thin] (1.6,4.7) -- (1.6,1.7);
\draw [lightgray,ultra thin] (1.3,5.5) -- (1.3,3.1);
\draw [lightgray,ultra thin] (1,5.5) -- (1,3.3);
\draw [lightgray,ultra thin] (0.7,5.5) -- (0.7,3.3);
\draw [lightgray,ultra thin] (0.4,5.5) -- (0.4,2.4);
\draw [lightgray,ultra thin] (1.3,2.7) -- (1.3,1.6);
\draw [lightgray,ultra thin] (1,2.7) -- (1,1.5);
\draw [lightgray,ultra thin] (0.7,2.7) -- (0.7,2);
\draw [lightgray,ultra thin] (0.4,1.7) -- (0.4,0);
\draw [lightgray,ultra thin] (0.7,1.7) -- (0.7,0);

\draw [lightgray,ultra thin] (0.1,0.3) -- (0.9,0.3);
\draw [lightgray,ultra thin] (0.1,0.6) -- (0.9,0.6);
\draw [lightgray,ultra thin] (0.1,0.9) -- (0.9,0.9);
\draw [lightgray,ultra thin] (0.1,1.2) -- (0.9,1.2);
\draw [lightgray,ultra thin] (0.1,1.5) -- (0.9,1.5);

\draw [lightgray,ultra thin] (0.9,1.9) -- (1.9,1.9);
\draw [lightgray,ultra thin] (0.9,2.2) -- (1.9,2.2);
\draw [lightgray,ultra thin] (0.1,2.5) -- (1.9,2.5);

\draw [lightgray,ultra thin] (0.1,3.5) -- (2.9,3.5);
\draw [lightgray,ultra thin] (1.1,3.2) -- (2.9,3.2);
\draw [lightgray,ultra thin] (1.8,2.9) -- (2.9,2.9);

\draw [lightgray,ultra thin] (0.1,4.5) -- (3.9,4.5);
\draw [lightgray,ultra thin] (0.1,4.16) -- (3.9,4.16);
\draw [lightgray,ultra thin] (0.1,3.85) -- (3.9,3.85);

\draw [lightgray,ultra thin] (0.1,4.85) -- (1.4,4.85);
\draw [lightgray,ultra thin] (0.1,5.2) -- (1.4,5.2);
\draw [lightgray,ultra thin] (2.3,5.2) -- (6,5.2);
\draw [lightgray,ultra thin] (2.3,4.85) -- (6,4.85);

\draw [->,thick] (7.4,3) -- node[above=2pt]{ } (10.4,3);

\begin{scope}[shift={(12,0.3)}]

\draw (9,3) node{$= \mathbf{p'}$};
\draw (0,2.5) -- (0,-0.6) -- (1,-0.6) -- (1,1.5) -- (2,1.5) -- (2,2.5);
\draw [densely dotted] (0,2.5) -- (0,3.5);
\draw [densely dotted] (2,2.5) -- (3,2.5) -- (3,3.5);
\draw (3,3.5) -- (4,3.5) -- (4,4.5) -- (5.5,4.5) -- (5.5,5.5) -- (0,5.5) -- (0,3.5);

\draw [thick, blue] (0,5.5) -- (5.5,5.5); 
\draw [thick, orange] (0,5.5) -- (0,3.5); 
\draw [thick, orange, densely dotted] (0,2.5) -- (0,3.5);
\draw [thick, orange] (0,-0.6) -- (0,2.5);

\draw[decoration={brace,raise=2pt},decorate] (0,5.5) -- node[above=2pt] {\small $n_1 + \dots + n_r$} (5.5,5.5);
\draw[decoration={brace,raise=2pt},decorate] (0,4.5) -- node[above=2pt] {\tiny $n_1 + \dots + n_{r-1}$} (4,4.5);
\draw[decoration={brace,raise=2pt},decorate] (0,2.5) -- node[above=2pt] {\tiny $n_1+n_2$} (2,2.5);
\draw[decoration={brace,raise=2pt},decorate] (0,1.5) -- node[above=2pt] {\tiny $n_1$} (1,1.5);

\draw[decoration={brace,raise=2pt},decorate] (5.5,5.5) -- node[right=2pt] {\small $m_r$} (5.5,4.5);
\draw[decoration={brace,raise=2pt},decorate] (4,4.5) -- node[right=2pt] {\small $m_{r-1}-m_r$} (4,3.5);
\draw[decoration={brace,raise=2pt},decorate] (2,2.5) -- node[right=2pt] {\small $m_2-m_3$} (2,1.5);
\draw[decoration={brace,raise=2pt},decorate] (1,1.5) -- node[right=2pt] {\small $m_1-m_2$} (1,-0.6);

\draw [lightgray,ultra thin] (4,5.5) -- (4,4.7);
\draw [lightgray,ultra thin] (4.3,5.5) -- (4.3,4.7);
\draw [lightgray,ultra thin] (4.6,5.5) -- (4.6,4.7);
\draw [lightgray,ultra thin] (4.9,5.5) -- (4.9,4.7);
\draw [lightgray,ultra thin] (5.2,5.5) -- (5.2,4.7);
\draw [lightgray,ultra thin] (3.7,5.5) -- (3.7,3.7);
\draw [lightgray,ultra thin] (3.4,5.5) -- (3.4,3.7);
\draw [lightgray,ultra thin] (3.1,5.5) -- (3.1,3.7);
\draw [lightgray,ultra thin] (2.8,5.5) -- (2.8,2.7);
\draw [lightgray,ultra thin] (2.5,5.5) -- (2.5,2.7);
\draw [lightgray,ultra thin] (2.2,4.7) -- (2.2,2.7);
\draw [lightgray,ultra thin] (1.9,4.7) -- (1.9,2.7);
\draw [lightgray,ultra thin] (1.6,4.7) -- (1.6,1.7);
\draw [lightgray,ultra thin] (1.3,5.5) -- (1.3,3.1);
\draw [lightgray,ultra thin] (1,5.5) -- (1,3.3);
\draw [lightgray,ultra thin] (0.7,5.5) -- (0.7,3.3);
\draw [lightgray,ultra thin] (0.4,5.5) -- (0.4,2.4);
\draw [lightgray,ultra thin] (1.3,2.7) -- (1.3,1.6);
\draw [lightgray,ultra thin] (1,2.7) -- (1,1.5);
\draw [lightgray,ultra thin] (0.7,2.7) -- (0.7,2);
\draw [lightgray,ultra thin] (0.4,1.7) -- (0.4,-0.5);
\draw [lightgray,ultra thin] (0.7,1.7) -- (0.7,-0.5);
\draw [lightgray,ultra thin] (0.1,-0.3) -- (0.9,-0.3);
\draw [lightgray,ultra thin] (0.1,0) -- (0.9,0);
\draw [lightgray,ultra thin] (0.1,0.3) -- (0.9,0.3);
\draw [lightgray,ultra thin] (0.1,0.6) -- (0.9,0.6);
\draw [lightgray,ultra thin] (0.1,0.9) -- (0.9,0.9);
\draw [lightgray,ultra thin] (0.1,1.2) -- (0.9,1.2);
\draw [lightgray,ultra thin] (0.1,1.5) -- (0.9,1.5);

\draw [lightgray,ultra thin] (0.9,1.9) -- (1.9,1.9);
\draw [lightgray,ultra thin] (0.9,2.2) -- (1.9,2.2);
\draw [lightgray,ultra thin] (0.1,2.5) -- (1.9,2.5);

\draw [lightgray,ultra thin] (0.1,3.5) -- (2.9,3.5);
\draw [lightgray,ultra thin] (1.1,3.2) -- (2.9,3.2);
\draw [lightgray,ultra thin] (1.8,2.9) -- (2.9,2.9);

\draw [lightgray,ultra thin] (0.1,4.5) -- (3.9,4.5);
\draw [lightgray,ultra thin] (0.1,4.16) -- (3.9,4.16);
\draw [lightgray,ultra thin] (0.1,3.85) -- (3.9,3.85);

\draw [lightgray,ultra thin] (0.1,4.85) -- (1.4,4.85);
\draw [lightgray,ultra thin] (0.1,5.2) -- (1.4,5.2);
\draw [lightgray,ultra thin] (2.3,5.2) -- (5.4,5.2);
\draw [lightgray,ultra thin] (2.3,4.85) -- (5.4,4.85);
\end{scope}
\end{tikzpicture}
\end{center}

We will often consider sums over all partitions of, e.g., a fixed number and with fixed length and therefore give the following definitions:
\begin{definition}
\label{part set def}
Define for every $N\in\N$ and $r\geq 1$ the set of partitions of $N$ of length $r$,
    \begin{align*}
    \mathcal{P}_{r}(N) := \left\{((m_1,\dots,m_r),(n_1,\dots,n_r))\in\N^r\times \N^r\, \middle|\, m_1 > \dots > m_r,\,  \sum_{j=1}^r {m_j n_j} =N\right\},
    \end{align*}
and with analogous notation,
    \begin{align*}
        \mathcal{P}_{\leq r}(N) := \bigcup\limits_{s=1}^r \mathcal{P}_{s}(N),\quad \mathcal{P}_{r} := \bigcup\limits_{N > 0} \mathcal{P}_r(N),\quad
        \mathcal{P}_{\leq r} := \bigcup\limits_{N > 0} \mathcal{P}_{\leq r}(N),\quad         \mathcal{P} := \bigcup\limits_{r\geq 1} \mathcal{P}_r.
    \end{align*}
\end{definition}
Remark at this point that the map $\rho:\mathcal{P}\rightarrow\mathcal{P},\ \mathbf{p}\mapsto\mathbf{p'}$ is an involution, and the restriction to one of the other sets in Def. \ref{part set def} is also an involution. By abuse of notation, the restricted maps will also be denoted by $\rho$.

\begin{theorem}[{\cite[Thm. 4.6]{Bri}}]
\label{qmzvs qser}
For every $\zeta_q\left(k_1,\dots,k_r;Q_1,\dots,Q_r\right)\in\mathcal{Z}_q$ with $r\in\N$, there are rational numbers $a_\mathbf{p}\in\Q$ for all partitions $\mathbf{p}\in\mathcal{P}_{\leq r}$ such that
\begin{align*}
    \zeta_q\left(k_1,\dots,k_r;Q_1,\dots,Q_r\right) = \sum\limits_{((m_1,\dots,m_{r'}),(n_1,\dots,n_{r'}))\in\mathcal{P}_{\leq r}} a_{\substack{m_1,\dots,m_{r'}\\ n_1,\dots,n_{r'}}} q^{m_1 n_1 + \dots + m_{r'} n_{r'}}.
\end{align*}
Moreover, these are polynomials in $m_1,\dots,m_{r'},n_1,\dots,n_{r'}$.
\end{theorem}

\begin{proof}
This follows from the expansion
\begin{align*}
    \frac{q^{m \ell}}{(1-q^m)^k} = q^{m \ell}\sum\limits_{n\geq 0} \binom{n+k-1}{k-1} q^{m n} = \sum\limits_{n\geq \ell} \binom{n-\ell + k-1}{k-1} q^{m n}
\end{align*}
for all $m\geq 1$, $0\leq \ell\leq k$ and from $\deg (Q_j)\leq k_j$, that all $Q_j$ have rational coefficients and from the fact that binomial coefficients in particular are rational numbers too. That the coefficients are polynomial in $m_1,\dots,m_{r'},n_1,\dots,n_{r'}$ follows also direct from this expansion.
\end{proof}

Theorem \ref{qmzvs qser} implies that for every element $S\in\mathcal{Z}_q$ there is a map $\textbf{a}:\mathcal{P}\longrightarrow \Q$ and rational $a_0$ such that
\begin{align*}
    S = a_0 + \sum\limits_{N \geq 1} \left(\sum\limits_{\mathbf{p}\in\mathcal{P}(N)} \textbf{a}(\mathbf{p})\right) q^N
\end{align*}
and all but finite many of the projections $\textbf{a}_r := \textbf{a}|_{\mathcal{P}_r}$, $r\geq 1$, are constant zero.

The mappings $\textbf{a}$ do not have to be unique, but we can find a nice one for each element $S\in\mathcal{Z}_q$ because they are polynomial:

\begin{theorem}[Thm. \ref{def Z_q2} ({\cite[Thm. 4.7]{Bri}})]
\label{Zq poly descr}
A $q$-series $S$ is in $\mathcal{Z}_q$ iff there exists $f=(f_r)_{r\geq 0}$ with\\ $f_r\in \Q[X_1,\dots,X_r,Y_1,\dots,Y_r]$ for $r\geq 1$ and $f_0\in\Q$ such that
\begin{enumerate*}
    \item $f_r\equiv 0$ for all but finite many $r$,
    \item  $
    S = f_0 + \sum\limits_{N\geq 1} \left(\sum\limits_{r\geq 1} \sum\limits_{(\textbf{m},\textbf{n})\in\mathcal{P}_r(N)} f_r(m_1,\dots,m_r,n_1,\dots,n_r)\right) q^N.$
\end{enumerate*}
\end{theorem}

\begin{proof}
Remark first that for every bi-bracket such an $f$ exists by the original definition of bi-brackets,
\begin{align*} 
    \g\bi{k_1,\dots,k_r}{d_1,\dots,d_r} :=
    \sum\limits_{\substack{m_1>\dots>m_r>0\\ n_1,\dots,n_r > 0}} \frac{m_1^{d_1}}{d_1!}\cdots \frac{m_r^{d_r}}{d_r!} \frac{n_1^{k_1 - 1}\dots n_r^{k_r -1}}{(k_1-1)!\dots (k_r - 1)!} q^{m_1 n_1 + \dots + m_r n_r}.
\end{align*}
There, $f_s\equiv 0$ for all $s$ except $s=r$, the depth of the bi-bracket. And $f_r$ is a monomial (up to a rational factor) in $\Q[X_1,\dots,Y_r]$. Indeed, since $d_i\geq 0,\, k_j\geq 1$ can take all values, $f_r$ coming from a bi-bracket can be every monomial in $\Q[X_1,\dots,Y_r]$. Furthermore, this holds for every $r\geq 1$.

Now, if $S\in\mathcal{Z}_q$, $S$ is a rational linear combination of bi-brackets since they span $\mathcal{Z}_q$. In this case, we see that $S$ indeed is of the desired shape since a possible $f$ is a finite rational linear combination of monomials by the remark above. Hence, in particular, it is a polynomial again.

Conversely, if $S$ is of the shape in the theorem, the monomials occurring in $f$ correspond to bi-brackets as remarked, i.e., $S$ is a rational linear combination of bi-brackets, an element of $\mathcal{Z}_q$.
\end{proof}

\begin{example}
By some straightforward calculation, we get
\begin{align*}
    &\, \zeta_q(1,0,2;X,1,1+X) = \sum\limits_{m_1>m_2>m_3>0} \frac{q^{m_1}}{1-q^{m_1}}\frac{1+q^{m_3}}{(1-q^{m_3})^2} 
    \\
    =&\, 2\sum\limits_{\substack{m_1>m_2>0\\ n_1,n_2> 0}} (m_1-m_2-1)(n_2 + 1) q^{m_1 n_1 + m_2 n_2} + \sum\limits_{\substack{m_1>0\\ n_1> 0}} \frac{(m_1 - 2)(m_1 - 1)}{2} q^{m_1 n_1}.
\end{align*}
In terminology of maps $\mathbf{a}:\mathcal{P}\rightarrow\Q$ we find now for $\zeta_q(1,0,2;X,1,1+X)$ a suitable map
\begin{align*}
    \textbf{a} : \mathcal{P} &\longrightarrow \Q,
    \\
    (\mathbf{m},\mathbf{n}) &\longmapsto \delta_{r=1} \frac{(m_1-2)(m_1-1)}{2} + \delta_{r=2} \cdot 2(m_1-m_2-1)(n_2+1).
\end{align*}
Especially, we see that Theorem \ref{Zq poly descr} applies here with
\begin{align*}
    f_1(m_1,n_1) := \frac{(m_1-2)(m_1-1)}{2},\quad f_2(m_1,m_2,n_1,n_2) := 2(m_1-m_2-1)(n_2+1)
\end{align*}
and $f_r :\equiv 0$ for $r>2$, which are all polynomial.
\end{example}

\begin{remark}
Functions like $\textbf{a}$ provide a direct connection from $q$MZVs to so-called $q$-brackets. For a function $\textbf{a}:\mathcal{P}\rightarrow \Q$, the $q$-bracket of $\textbf{a}$ is defined as
\begin{align*}
    \left\langle \textbf{a}\right\rangle_q := \frac{\sum\limits_{\lambda\in\mathcal{P}} \textbf{a}(\lambda) q^{|\lambda|}}{\sum\limits_{\lambda\in\mathcal{P}} q^{|\lambda|}}.
\end{align*}
They were first introduced by Bloch and Okounkov in \cite{BO} and are of interest in current research since, under certain conditions on $\textbf{a}$, the $q$-bracket is quasi-modular. Recall at this point that every quasi-modular form is, in particular, an element of $\mathcal{Z}_q$.

The exact connection between $q$MZVs and $q$-brackets will be described in \cite{BvI}. For further research details on $q$-brackets, we refer to the works by Zagier (\cite{Za2}) and van Ittersum (\cite{vIt}).
\end{remark}

\begin{lemma}[{\cite[Lem. 4.9]{Bri}}]
\label{part rel}
For all $r,N\geq 1$ and maps $\textbf{a}_r:\mathcal{P}_r(N)\rightarrow \Q$ we have the equation
\begin{align}
\label{rho}
    \sum\limits_{\mathbf{p}\in\mathcal{P}_r(N)} \textbf{a}_r(\mathbf{p}) = \sum\limits_{\mathbf{p}\in\mathcal{P}_r(N)} \textbf{a}_r(\rho(\mathbf{p})).
\end{align}
\end{lemma}
\begin{proof}
The map $\rho$ is an involution on $\mathcal{P}_{r}(N)$.
\end{proof}
This lemma is important when considering duality relations among $q$MZVs like Schlesinger--Zudilin duality. For details, see \cite[Lem. 4.13]{Bri}.

\subsection{Bi-brackets}

For every bi-bracket $\g\bi{k_1,\dots,k_r}{d_1,\dots,d_r}$ the coefficient of $q^N$ can be easily derived by the original definition:
\begin{align*}
    \g\bi{k_1,\dots,k_r}{d_1,\dots,d_r}
    &=\,
    \sum\limits_{\substack{m_1>\dots>m_r>0\\ n_1,\dots,n_r > 0}} \frac{m_1^{d_1}}{d_1!}\cdots \frac{m_r^{d_r}}{d_r!} \frac{n_1^{k_1 - 1}\dots n_r^{k_r -1}}{(k_1-1)!\dots (k_r - 1)!} q^{m_1 n_1 + \dots + m_r n_r}
    \\
    &=\,
    \frac{1}{\prod\limits_{j=1} d_j!(k_j - 1)!}
    \sum\limits_{N > 0} \left(\sum\limits_{(\mathbf{m},\mathbf{n})\in \mathcal{P}_r(N)} m_1^{d_1}\cdots m_r^{d_r} n_1^{k_1 - 1}\cdots n_r^{k_r -1}\right) q^{N}.
\end{align*}

Lemma \ref{part rel} gives an explicit expression of the so-called \emph{partition relation} (\cite{Ba2}):

\begin{lemma}[{\cite[Lem. 4.11]{Bri}}]
\label{prel direct}
For all $r\geq 1$, $d_1,\dots,d_r\geq 0,\, k_1,\dots,k_r\geq 1$, we have
\begin{align*}
    \g\bi{k_1,\dots,k_r}{d_1,\dots,d_r} 
    =
    &\sum\limits_{\substack{0\leq k_i'\leq k_i-1\\ d_{i,j}'\geq 0,\, 1\leq i\leq r,\, 1\leq j\leq r-i+1\\ d_{i,1}'+\dots+d_{i,r-i+1}' = d_i}}
    \prod\limits_{j=1}^r \frac{\left(d'_{1,j}+\dots + d'_{r-j+1,j}\right)! \left(k_{r-j+1}'+k_{r-j+2}-1-k_{r-j+2}'\right)!}{d_j!\left(k_j - 1\right)!}
    \\
    &\quad \times\binom{d_j}{d_{j,1}',\dots,d_{j,r-j+1}'} \binom{k_j - 1}{k_j'} 
     \g\bi{d'_{1,1}+\dots +d'_{1,r},\dots,d'_{r-1,1} + d'_{r-1,2},d'_{r,1}}{k'_{r}-k'_{r+1}-1 + k_{r+1},\dots,k'_{1}-k'_{2}-1 + k_{2}}.
\end{align*}
with $k_{r+1}:=k_{r+1}':=0$.
\end{lemma}

Interesting in the context of bi-brackets and Theorem \ref{Zq poly descr} is the following reformulation of Bachmann's conjecture which says that brackets and bi-brackets span the same space:

\begin{conjecture}[Refinement of {\cite[Conj. 4.3]{Ba4}}]
Let $P\in\Q[X_1,\dots,X_r,Y_1,\dots,Y_r]$, $r\geq 1$, be a polynomial. Then there exists $Q=(Q_j)_{j\geq 1}$ with $Q_j\in\Q[X_1,\dots,X_j]$ and $Q_j\equiv 0$ for all but finite many $j$ such that for every $N\geq 1$, with $M:=r+\sum\limits_{i=1}^r \deg_{Y_i}(P)$, we have 
\begin{align*}
    \sum\limits_{(\mathbf{m},\mathbf{n})\in\mathcal{P}_r(N)} P(m_1,\dots,m_r,n_1,\dots,n_r)
    =\, 
    \sum\limits_{j = 1}^{M} \sum\limits_{(\mathbf{m},\mathbf{n})\in\mathcal{P}_j(N)} Q_j(n_1,\dots,n_r).
\end{align*}
\end{conjecture}

\subsection{SZ-model}
Consider some SZ-admissible index ${\mathbf{k}} = \left(k_1 + 1,\{0\}^{d_1},\dots,k_r + 1,\{0\}^{d_r}\right)$ and obtain
\begin{align*}
    \sz({\mathbf{k}}) 
    =\,  \sum\limits_{\substack{m_1 > \dots > m_r > 0\\ n_1,\dots, n_r > 0}} \prod\limits_{j=1}^r \binom{m_j-m_{j+1} - 1}{d_j}\binom{n_j - 1}{k_j} q^{m_j n_j}.
\end{align*}

We can interpret the coefficient of $q^N$ as the number of partitions of $N$ with rows and columns in the Young diagram marked in some way:
\begin{proposition}[{\cite[Prop. 4.14]{Bri}}]
\label{sz young prop}
    The coefficient of $q^N$ in $\sz(\mathbf{k})$ is the number of partitions of $N$ with exactly $r$ parts, where $d_i$ rows of the $i$-th part without the last row in the corresponding Young diagram are marked, such as $k_j$ columns lying between the $j$-th and $(j+1)$-th rightmost corner of the Young diagram for all $1\leq i,j\leq r$.
\end{proposition}

\begin{example}
\label{ex young sz}
For $N=126$ one of the in Proposition \ref{sz young prop} described marked partitions with exactly $r=3$ parts and with $k_1 = 2,\, k_2=k_3=1,\ d_1=2,\, d_2=1,\, d_3=3$ is the following:

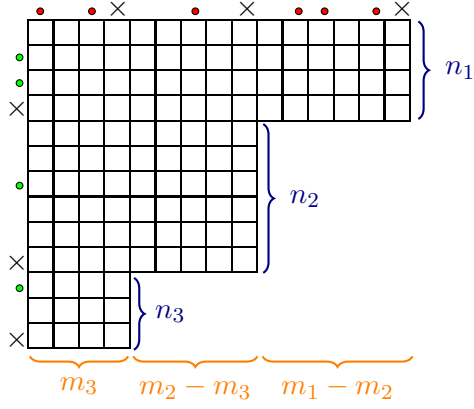
\begin{center}
{\ensuremath{
\ytableausetup{smalltableaux}
\begin{tikzpicture}[inner sep=0in,outer sep=0in,scale=0.34]
\node (n)
{\ydiagram{15, 15, 15, 15, 9, 9, 9, 9, 9, 9, 4, 4, 4}};
\foreach \x in {1, 3, 7, 11, 12, 14}
		\filldraw[fill=red] (n.north west)+(0.5+\x-1,0.3) circle [radius=0.13cm];
		\foreach \y in {4, 9, 15}
		\node at ($(n.north west)+(0.5+\y-1,0.4)$) () {$\times$};
\foreach \x in {4, 10, 13}
	\node at ($(n.north west)+(-0.4,-0.5-\x+1)$) () {$\times$};
	\foreach \y in {2, 3, 7, 11}	
	\filldraw[fill=green] (n.north west)+(-0.3,-0.5-\y+1) circle [radius=0.13cm];
	
\draw[thick,darkblue,decorate,decoration={brace,amplitude=4pt}] ($(n.north east)+(0.3,-0.1)$) -- ($(n.north east) + (0.3,-3.9)$) node[midway,right,xshift=10pt] {$n_1$};

\draw[thick,darkblue,decorate,decoration={brace,amplitude=4pt}] ($(n.north east)+(-5.7,-4.1)$) -- ($(n.north east) + (-5.7,-9.9)$) node[midway,right,xshift=10pt] {$n_2$};

\draw[thick,darkblue,decorate,decoration={brace,amplitude=4pt}] ($(n.north east)+(-10.7,-10.1)$) -- ($(n.north east) + (-10.7,-12.9)$) node[midway,right,xshift=8pt] {$n_3$};

\draw[thick,orange,decorate,decoration={brace,amplitude=4pt}] ($(n.south west) + (3.9,-0.3)$) -- ($(n.south west)+(0.1,-0.3)$) node[midway,below,yshift=-8pt] {$m_3$};

\draw[thick,orange,decorate,decoration={brace,amplitude=4pt}] ($(n.south west) + (8.9,-0.3)$) -- ($(n.south west)+(4.1,-0.3)$) node[midway,below,yshift=-8pt] {$m_2-m_3$};

\draw[thick,orange,decorate,decoration={brace,amplitude=4pt}] ($(n.south west) + (14.9,-0.3)$) -- ($(n.south west)+(9.1,-0.3)$) node[midway,below,yshift=-8pt] {$\small{m_1-m_2}$};
\end{tikzpicture}
}}
\end{center}
The crosses $\times$ stand for the corresponding row/column not being allowed to be colored. That there is in every part a fixed row/column (we always fix the lowest row/rightest column) comes from the $-1$'s in the binomial coefficients that we consider in the coefficient of $q^N$ in $\sz(\mathbf{k})$.
\end{example}

\begin{proof}[Proof (of Lemma \ref{sz young prop})]  
Considering the index of the first sum of the coefficient of $q^N$ in $\sz(\mathbf{k})$, we get that it is just the number of partitions of $N$ with exactly $r$ parts, every partition counted with the respective multiplicity, given as the product of binomial coefficients we have seen.

Now, given such a partition of $N$, the $j$-the part consists of $n_j$ rows, why there are $\tbinom{n_j -1}{k_j}$ ways of marking $d_j$ of the rows of the $j$-th part without the last one. Since those markings in every part are independent of the markings in the other parts, this rows coloring gives a multiplicity of
\begin{align*}
    \binom{n_1 - 1}{k_1}\cdots \binom{n_r - 1}{k_r} 
\end{align*}
of the given partition.

For the column markings we use similar arguments: Since the $j$-th part of the Young diagram has length $m_j$ and the $(j+1)$-th has length $m_{j+1}$, there are $m_j - m_{j+1} - 1$ columns between the $j$-th and $(j+1)$-th corner, counted from the right. Hence, with marking $d_j$ of them, we get an additional factor $\tbinom{m_j - m_{j+1} - 1}{d_j}$ for the multiplicity of the given partition (for every $1\leq j\leq r$) and so exactly the coefficient of $q^N$ of $\sz(\mathbf{k})$.\end{proof}

\begin{remark}
By transposing Young diagrams together with the markings, we get a new generating series of specific marked partitions. One can verify that this generating series is the SZ-$q$MZV of SZ-dual index. Hence, using marked partition, we deduce SZ-duality (\cite[§4.2]{Bri}).
\end{remark}

\subsection{BZ-model}

For admissible
${\mathbf{k}} = \left(k_1 + 1,\{1\}^{d_1 - 1},\dots,k_r + 1,\{1\}^{d_r - 1}\right)$, i.e., $k_1,\dots,k_r,d_1,\dots,d_r\geq 1$, we compute
\begin{align*}
    \bz({\mathbf{k}}) = \sum\limits_{\substack{m_1 > n_1 > \dots > n_{d_1-1} > m_2 > \dots\\ > m_r > \dots > n_{d_1  + \dots + d_r -r} >0\\ j_1,i_1,\dots,i_{d_1-1},j_2\\ \dots,i_{d_1  + \dots + d_r -r}\geq 0}} & \binom{j_1}{k_1}\cdots \binom{j_r}{k_r} q^{m_1 j_ 1 + \dots + m_r j_r + n_1 i_1 + \dots + n_{d_1+\dots + d_r - r} i_{d_1 + \dots + d_r -r}}.
\end{align*}

The coefficient of $q^N$ is the number of partitions of $N$, where some (exactly $j_1\geq 0$) rows of the 1st, some (exactly $j_2\geq 0$) of the $(d_1+1)$-th part in the Young diagram are marked and so on.

It is useful to avoid the partitions where some $j_i = 0$ for clearness. Hence, we introduce also markings of the columns of the Young diagrams as follows:

\begin{proposition}[{\cite[Prop. 4.19]{Bri}}]
\label{prop bz part}
The coefficient of $q^N$ in $\bz(\mathbf{k})$ corresponds to the number of partitions of $N$, where the corresponding Young diagram is split up into $r$ sub-Young diagrams with at most $d_1,\dots,d_r$ parts each. We mark $k_j$ rows in the first part of sub-Young diagram $j$ for each $1\leq j\leq r$. Furthermore, we mark all columns containing corners and some of the others such that the number of colored columns, only belonging to sub-Young diagram $j$, in total is $d_j$ for each $1\leq j\leq r$.
\end{proposition}

\begin{proof}
We obtain the split up into $r$ sub-Young diagrams by the first line of the sum index of the coefficient of $q^N$. Also, the row markings are self-explaining when looking at the summand of our coefficient of $q^N$.

The marked columns represent the indices (from right to left) of shape $j_\ell$ and $i_\ell$. If a marked column is not the rightmost one of a part, this corresponds to whether the corresponding multiplicity $i_\ell$ is zero. In this case, there is no $n_\ell$-part in the partition, which is, on the one hand, the reason for having exactly $d_j$ marked columns that belong to sub-Young diagram $j$ for each $1\leq j\leq r$ and on the other hand, it is the reason why we have at most (and not exact) $d_i$ parts in sub-Young diagram $i$.\end{proof}

\begin{example}
The following marked partition of $N=118$ has $r=2$ sub-Young diagrams and is assigned to the index $\mathbf{k} = (3,1,1,1,3,1,1)$, i.e., $k_1 = k_2 = 2,\, d_1 = 4,\, d_2 = 3$:

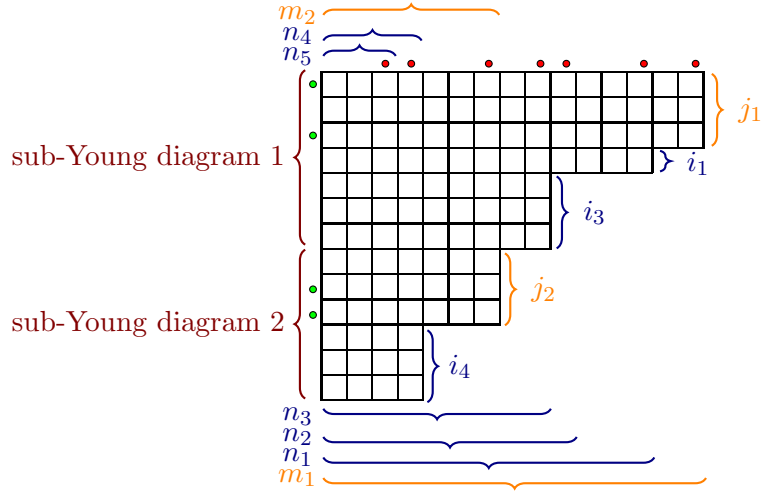
\begin{center}
{\ensuremath{
\ytableausetup{smalltableaux}
\begin{tikzpicture}[inner sep=0in,outer sep=0in,scale=0.34]

\node (n)
{\ydiagram{15, 15, 15, 13, 9, 9, 9, 7, 7, 7, 4, 4, 4}};
\foreach \x in {3, 4, 7, 9, 10, 13, 15}
		\filldraw[fill=red] (n.north west)+(0.5+\x-1,0.3) circle [radius=0.13cm];
\foreach \y in {1, 3, 9, 10}	
	\filldraw[fill=green] (n.north west)+(-0.3,-0.5-\y+1) circle [radius=0.13cm];
	
\draw[thick,orange,decorate,decoration={brace,amplitude=4pt}] ($(n.north east)+(0.3,-0.1)$) -- ($(n.north east) + (0.3,-2.9)$) node[midway,right,xshift=10pt] {$j_1$};

\draw[thick,darkblue,decorate,decoration={brace,amplitude=4pt}] ($(n.north east)+(-1.7,-3.1)$) -- ($(n.north east) + (-1.7,-3.9)$) node[pos = 0.6,right,xshift=10pt] {$i_1$};

\draw[thick,darkblue,decorate,decoration={brace,amplitude=4pt}] ($(n.north east)+(-5.7,-4.1)$) -- ($(n.north east) + (-5.7,-6.9)$) node[midway,right,xshift=10pt] {$i_3$};

\draw[thick,orange,decorate,decoration={brace,amplitude=4pt}] ($(n.north east)+(-7.7,-7.1)$) -- ($(n.north east) + (-7.7,-9.9)$) node[midway,right,xshift=10pt] {$j_2$};

\draw[thick,darkblue,decorate,decoration={brace,amplitude=4pt}] ($(n.north east)+(-10.7,-10.1)$) -- ($(n.north east) + (-10.7,-12.9)$) node[midway,right,xshift=8pt] {$i_4$};

\draw[thick,darkblue,decorate,decoration={brace,amplitude=4pt}] ($(n.north west)+(0.1,0.6)$) -- ($(n.north west) + (2.9,0.6)$); 

\node[text=darkblue] at ($(n.north west) + (-0.9,0.6)$) () {$n_5$};

\draw[thick,darkblue,decorate,decoration={brace,amplitude=4pt}] ($(n.north west)+(0.1,1.2)$) -- ($(n.north west) + (3.9,1.2)$); 

\node[text=darkblue] at ($(n.north west) + (-0.9,1.4)$) () {$n_4$};

\draw[thick,orange,decorate,decoration={brace,amplitude=4pt}] ($(n.north west)+(0.1,2.2)$) -- ($(n.north west) + (6.9,2.2)$); 

\node[text=orange] at ($(n.north west) + (-0.9,2.2)$) () {$m_2$};

\draw[thick,darkblue,decorate,decoration={brace,amplitude=4pt}] ($(n.south west) - (-8.9,0.3)$) -- ($(n.south west)-(-0.1,0.3)$); 

\node[text=darkblue] at ($(n.south west) - (0.9,0.6)$) () {$n_3$};

\draw[thick,darkblue,decorate,decoration={brace,amplitude=4pt}] ($(n.south west) - (-9.9,1.4)$) -- ($(n.south west)-(-0.1,1.4)$); 

\node[text=darkblue] at ($(n.south west) - (0.9,1.4)$) () {$n_2$};

\draw[thick,darkblue,decorate,decoration={brace,amplitude=4pt}] ($(n.south west) - (-12.9,2.2)$) -- ($(n.south west)-(-0.1,2.2)$); 

\node[text=darkblue] at ($(n.south west) - (0.9,2.2)$) () {$n_1$};

\draw[thick,orange,decorate,decoration={brace,amplitude=4pt}] ($(n.south west) - (-14.9,3.0)$) -- ($(n.south west)-(-0.1,3.0)$); 

\node[text=orange] at ($(n.south west) - (0.9,3.0)$) () {$m_1$};

\draw[thick,color=darkred,decorate,decoration={brace,amplitude=4pt}] ($(n.south west) - (0.6,-0.1)$) -- ($(n.south west)-(0.6,-5.9)$) node[midway,left,xshift=-7.5] () {sub-Young diagram 2}; 

\draw[thick,color=darkred,decorate,decoration={brace,amplitude=4pt}] ($(n.south west) + (-0.6,6.1)$) -- ($(n.south west)+(-0.6,12.9)$) node[midway,left,xshift=-7.5] () {sub-Young diagram 1};

\end{tikzpicture}}}
\end{center}
Remark that $n_3$ occurs with multiplicity $i_2 = 0$, such as $n_5$ with multiplicity $i_5 = 0$ which is the reason that the columns corresponding to $n_3$ and $n_5$ respectively are marked but contain no corner of the Young diagram.
\end{example}

The following statement is now equivalent to BZ-duality (see \cite[§4.3]{Bri}):

\begin{theorem}[{\cite[4.18]{Bri}}]
\label{5.10}
For all $r\geq 1$, $k_1,\dots,k_r,d_1,\dots,d_r\geq 1$ and $N\geq 1$ we have:
\begin{align*}
    &
    \sum\limits_{\substack{m_1 > n_1 > \dots > n_{d_1-1} > m_2 > \dots > m_r > \dots > n_{d_1  + \dots + d_r -r} > 0\\ j_1,i_1,\dots,i_{d_1-1},j_2,\dots,i_{d_1  + \dots + d_r -r}\geq 0\\ m_1 j_1 + \dots + m_r j_r + n_1 i_1 + \dots + n_{d_1  + \dots + d_r -r} i_{d_1  + \dots + d_r -r} = N}} \binom{j_1}{k_1}\cdots \binom{j_r}{k_r}
    \\
    =\, &
    \sum\limits_{\substack{m_1 > n_1 > \dots > n_{k_r-1} > m_2 > \dots > m_r > \dots > n_{k_r  + \dots + k_1 -r} > 0\\ j_1,i_1,\dots,i_{k_r-1},j_2,\dots,i_{k_r  + \dots + k_1 -r}\geq 0\\ m_1 j_1 + \dots + m_r j_r + n_1 i_1 + \dots + n_{k_r  + \dots + k_1 -r} i_{k_r  + \dots + k_1 -r} = N}} \binom{j_1}{d_r}\cdots \binom{j_r}{d_1}
\end{align*}
Moreover, this is equivalent to the duality of Bradley--Zhao $q$MZVs.
\end{theorem}

\subsection{Partition function, \emph{q}MZVs and conjugacy classes}

When studying SZ-$q$MZVs and special values of them, we get a connection to the partition numbers:

\begin{lemma}[{\cite[Lem. 4.26]{Bri}}]
\label{part-SZ}
Let be $p_N$ the number of partitions of $N$, then
\begin{align*}
    \sum\limits_{r\geq 1} \sz(\{1\}^r) = \sum\limits_{r\geq 1} \g\bi{\{1\}^r}{\{0\}^r} = \sum\limits_{N\geq 1} p_N q^N.
\end{align*}
\end{lemma}
\begin{proof}  The first equality is clear by the definition of bi-brackets. We consider now the left side first:

\begin{align*}
    \sum\limits_{r\geq 0} \sz(\{1\}^r) = \sum\limits_{r\geq 0} \sum\limits_{m_1>\dots > m_r > 0} \frac{q^{m_1}}{1-q^{m_1}}\cdots \frac{q^{m_r}}{1-q^{m_r}} 
    =\sum\limits_{r\geq 0} \sum\limits_{\substack{m_1>\dots > m_r > 0\\ n_1,\dots,n_r > 0}} q^{m_1 n_1 + \dots + m_r n_r}.
\end{align*}
The coefficient of some $q^N$ here is the sum over all $r\in\mathbb{N}_0$, where we sum the number of all partitions of $N$ with exact $r$ different parts, i.e., the number of partitions of $N$, $p_N$.\end{proof}

The partition function also occurs in other contexts than $q$MZVs, namely when considering equivalence classes of the symmetric group $\mathcal{S}_n$. For more details, we refer to \cite[§4]{FH}.

\begin{lemma}
\label{S_n-part}
Partitions of $n\in\N$ and conjugacy classes of $\mathcal{S}_n$ are in 1:1-correspondence. In particular, the number of conjugacy classes of $\mathcal{S}_n$ is $p_n$.
\end{lemma}
\begin{proof}  Write every $\sigma\in\mathcal{S}_n$ as union of cycles. The length of the cycles form a partition of $n$. Since a conjugacy class $[\sigma]$ of $\mathcal{S}_n$ is uniquely determined by the lengths of cycles of $\sigma$ - conjugacy means only to rename the elements $1,\dots,n$, but not to change the structure of $\sigma$ - the claim follows.\end{proof} 

\begin{example}
The conjugacy class of $\sigma = (1\, 4\, 3)(2\, 6)(5\, 7)\in\mathcal{S}_7$ corresponds to the partition \begin{align*} 
\ydiagram{3,2,2}\ \ \text{of } 3+2+2=7.
\end{align*}
\end{example}

\begin{remark}
Lemmas \ref{part-SZ} and \ref{S_n-part} give a remarkable connection between the number of conjugacy classes of $\mathcal{S}_n$ and SZ-$q$MZVs. More precisely, fixing $r$ and $n$, the coefficient of $q^n$ in $\sz(\{1\}^r)$ is the number of conjugacy classes of $\mathcal{S}_n$ with cycles of exactly $r$ different lengths.
\end{remark}

This remark should be the motivation for the following theorem:

\begin{theorem}[{\cite[Thm. 4.30]{Bri}}]
\label{c OOZ GLn}
Let $K$ be a finite field with $c$ elements. Then we have
\begin{align*}
    \mathrm{G}_K := \sum\limits_{n\geq 0} a_{n,K} q^n = \sum\limits_{r\geq 0} (c-1)^r\ooz(\{1\}^r),
\end{align*}
where $a_{n,K}$ is the number of conjugacy classes of $\text{GL}(n,K)$ and with $a_{0,K} := 1$ for every field $K$.
\end{theorem}
The proof can be visualized with marked partitions. For details and connected results about several representations of the numbers $a_{n,K}$, we refer to \cite[§4.6]{Bri}.


\begin{thebibliography}{99\kern\bibindent}
	        \bibitem[Ba1]{Ba1} H. Bachmann: {\itshape Multiple Zeta-Werte und die Verbindung zu Modulformen durch Multiple Eisensteinreihen}, Masterthesis (2012).
	        
	        \bibitem[Ba2]{Ba2} H. Bachmann: {\itshape Multiple Eisenstein series and q-analogues of multiple zeta values}, PhD Thesis (2015).
	        
	        \bibitem[Ba3]{Ba3}  H.~Bachmann:	        
	         {\itshape  Double shuffle relations for q-analogues of Multiple Zeta Values, their derivatives and the connection to multiple Eisenstein series}, RIMS Kôkyûroku {\textbf{2017}} (2015), 22--43.
	        
	        \bibitem[Ba4]{Ba4}  H.~Bachmann:
	        {\itshape The algebra of bi-brackets and regularized multiple Eisenstein series}, J. Number Theory, {\textbf{200}} (2019), 260--294.
	        
	        \bibitem[Ba5]{Ba5} H.~Bachmann: {\itshape Multiple Eisenstein series and q-analogues of multiple zeta values}, Periods in Quantum Field Theory and Arithmetic, Springer Proceedings in Mathematics \& Statistics \textbf{314} (2020),  173--235.
	        
	        \bibitem[Ba6]{Ba6} H. Bachmann: Lectures on \textit{Multiple Zeta Values and Modular Forms}, Nagoya University (2020).
	        
	        \bibitem[Ba7]{Ba7} H. Bachmann: private communication (2021).
	       
	        
	        
	        \bibitem[BK1]{BK1} H. Bachmann, U. K\"uhn: {\itshape{The algebra of generating functions for multiple divisor sums and applications to multiple zeta values}}, Ramanujan J. \textbf{40} (2016), 605--648.
	        		
			\bibitem[BK2]{BK2}  H.~Bachmann, U.~K\"uhn:
			{\itshape A dimension conjecture for q-analogues of Multiple Zeta Values}, Periods in Quantum Field Theory and Arithmetic, Springer Proceedings in Mathematics \& Statistics {\textbf{314}} (2020),  237--258.
			
\bibitem[BvI]{BvI} H. Bachmann, J.-W. van Ittersum: \textit{Partitions, Multiple Zeta Values and the $q$-bracket} (in preparation).	

\bibitem[BO]{BO} S. Bloch, A. Okounkov: {\itshape The character of the infinite wedge representation}, Adv.
Math., \textbf{149}(1):1–60, 2000.


\bibitem[Bra]{Bra} D. Bradley: {\itshape{Multiple q-zeta values}}, J. of Algebra {\textbf{283}} (2005), 752--798.

\bibitem[Bri]{Bri} B. Brindle: {\itshape{Dualities of q-analogues of multiple zeta values}}, Master thesis, University of Hamburg, \url{https://sites.google.com/view/benjamin-brindle/start}, 2021.

\bibitem[BK]{BK} D. Broadhurst, D. Kreimer: \textit{Association of multiple zeta values with positive knots
via Feynman diagrams up to 9 loops}, Phys. Lett. B \textbf{393} (1997), 403--412.

\bibitem[Bro]{Bro} F. Brown: {\itshape Mixed Tate motives over $\Z$}, Ann. of Math., \textbf{175}(2) (2012), 949--976.

\bibitem[BF]{BF} J. Burgos Gil and J. Fresán: {\textit{Multiple Zeta Values: From Numbers to Motives}}, to appear.

\bibitem[CEM]{CEM}  J. Castillo Medina, K. Ebrahimi-Fard, D. Manchon: {\textit{Unfolding the double shuffle structure of q-multiple zeta values}}, Bull. Aust.
Math. Soc., {\textbf{91}}(3) (2015), 368--388.

\bibitem[DG]{DG} P. Deligne, A. Goncharov: {\itshape Groupes fondamentaux motiviques de Tate
mixte}: Ann. Sci. École Norm. Sup. (4), \textbf{38}(1) (2005), 1--56.
	
	
\bibitem[EMS]{EMS} 		K. Ebrahimi-Fard, D. Manchon, J. Singer:	{\itshape Duality and (q-)Multiple Zeta Values}, Adv. Math. {\textbf{298}} (2016), 254--285.


\bibitem[FH]{FH} W. Fulton, J. Harris: \textit{Representation Theory - A First Course}, Graduate Texts in Mathematics \textbf{129}, Springer, New York, NY (2004).

		
\bibitem[GKZ]{GKZ} H. Gangl, M. Kaneko, D. Zagier: {\itshape{Double zeta values and modular forms}}, Automorphic forms and zeta functions, Tokyo, Japan, World Scientific (2006).


\bibitem[Hof]{Hof} M. Hoffman: {\itshape Quasi-shuffle products}, J. Algebraic Combin. {\textbf{11}} (2000), 49--68.


\bibitem[IKZ]{IKZ} K.~Ihara, M.~Kaneko,  D.~Zagier:  {\itshape Derivation and double shuffle relations for Multiple Zeta Values}, Compositio Math. {\textbf{142}} (2006), 307--338.

\bibitem[KZ]{KZ}  M. Kaneko, D. Zagier: {\itshape{A generalized Jacobi theta function and quasimodular forms}},
The moduli space of curves, Progr. Math. \textbf{129} (1995), 165--172.

\bibitem[KKW]{KKW} M. Kaneko, N. Kurokawa, M. Wakayama: {\itshape{A variation of Eulers approach to values of the Riemann zeta function}}, Kyushu J. Math. {\textbf{57}} (2003), 175--192.




\bibitem[OOZ]{OOZ} Y. Ohno, J. Okuda, W. Zudilin: {\textit{Cyclic q-MZSV sum}}, J. Number Theory {\textbf{132}} (2012), 144--155.

\bibitem[Oko]{Oko} A. Okounkov: {\itshape{Hilbert schemes and multiple $q$-zeta values}}, Functional Analysis and Its Applications {\textbf{48}} (2014), 138--144.

\bibitem[OT]{OT} J. Okuda, Y. Takeyama: {\itshape On relations for the multiple q-zeta values}, Ramanujan J. \textbf{14}(3) (2007), 379--387.


\bibitem[Sch]{Sch} K.-G. Schlesinger: {\itshape{Some remarks on q-deformed multiple polylogarithms}}, \href{https://arxiv.org/pdf/math/0111022.pdf}{arXiv:math/0111022.}


\bibitem[SY]{SY} S. Seki, S. Yamamoto: {\itshape A new proof of the duality of Multiple Zeta Values and its generalizations}, Int. J. of Number Theory, {\textbf{15}}(6) (2019), 1261--1265.

\bibitem[SQ]{SQ} Z. Shen, Z. Qin: {\itshape Hilbert schemes of points and quasi-modularity}  PAMQ \textbf{16} (5) (2020), 1673--1706.



\bibitem[Sin]{Sin} J. Singer: {\itshape q-Analogues of Multiple Zeta Values and their application in renormalization}, Dissertation, Erlangen-N\"urnberg University (2017).

\bibitem[Sta]{Sta}  R. Stanley: {\itshape Irreducible symmetric group characters of rectangular shape}, Sem. Lothar. Combin. \textbf{50} (2003/2004).
			
\bibitem[Tak]{Tak} Y. Takeyama: {\itshape The algebra of a q-analogue of multiple harmonic series}, SIGMA {\textbf{9}} (2013), 1--15.


\bibitem[Ter]{Ter} T. Terasoma: {\itshape Mixed Tate motives and multiple zeta values}, Invent. Math. \textbf{149}(2) (2002), 339--369.


\bibitem[vIt]{vIt} J.-W. van Ittersum: {\itshape A symmetric Bloch–Okounkov theorem}, Res. Math. Sci. \textbf{8}(2) (2021).

\bibitem[Wal]{Wal} M. Waldschmidt: Lectures on Multiple Zeta Values, imsc 2011. \url{https://webusers.imj-prg.fr/~michel.waldschmidt/articles/pdf/MZV2011IMSc.pdf}, 2012.



\bibitem[Za1]{Za1} D. Zagier: {\itshape Values of Zeta Functions and Their Applications}, First European Congress of Mathematics Paris (1994).

\bibitem[Za2]{Za2} D. Zagier: {\itshape Partitions, quasimodular forms, and the Bloch-Okounkov theorem}, Ramanujan Journal \textbf{41} (2016), 345--368.

\bibitem[Zh1]{Zh1} J. Zhao: {\itshape{Multiple q-zeta functions and multiple q-polylogarithms}}, Ramanujan J., {\textbf{14}}(2) (2007), 189--221.

\bibitem[Zh2]{Zh2} J. Zhao: {\itshape{Uniform approach to double shuffle and duality relations of various q-Analogs of Multiple Zeta Values via Rota--Baxter algebras}}, In J. Burgos Gil, K. Ebrahimi-Fard, H. Gangl (Eds.), Periods in quantum field theory and arithmetic. Cham: Springer (2014), 259--292.


\bibitem[Zu1]{Zu1} W. Zudilin: {\itshape{Algebraic relations for Multiple Zeta Values}}, Russian Math. Surveys {\textbf{58}}(1) (2003), 1–29.

\bibitem[Zu2]{Zu2} W. Zudilin: {\itshape{Multiple $q$-Zeta Brackets}}, Mathematics {\textbf{3}}(1) (2015), 119–130.
   \end{thebibliography}
\end{document}